\documentclass[a4paper,11pt]{amsart}
\usepackage[utf8]{inputenc}
\usepackage{a4wide}
\usepackage{amsmath}
\usepackage{amsthm}
\usepackage{amsfonts}
\usepackage{amssymb}
\usepackage{afterpage}
\usepackage{float}
\usepackage[all]{xy}
\usepackage{tikz}
\usepackage{bbm}
\usepackage{color}

\theoremstyle{plain}
\newtheorem{theorem}{Theorem}[section]
\newtheorem{prop}[theorem]{Proposition}
\newtheorem{lem}[theorem]{Lemma}
\newtheorem{corol}[theorem]{Corollary}

\newtheorem{defi}[theorem]{Definition}

\newtheorem{rmq}[theorem]{Remark}
\newtheorem{exmp}[theorem]{Example}

\def\<{\left<}
\def\>{\right>}

\def\d{{\partial}}

\def\ens#1{\left\{ #1 \right\}}

\def\fl{{\longrightarrow}\,}

\def\A{{\mathbf{A}}}
\def\B{{\mathbf{B}}}
\def\C{{\mathbf{C}}}
\def\E{{\mathbf{E}}}
\def\T{{\mathbf{T}}}
\def\W{{\mathbf{W}}}

\def\qA{{\mathbf{A}^{\otimes}}}

\def\qE{{\mathbf{E}^{\otimes}}}
\def\qT{{\mathbf{T}^{\otimes}}}
\def\qW{{\mathbf{W}^{\otimes}}}

\def\S{{\mathbf{S}}}
\def\M{{\mathbf{M}}}

\def\SM{(\S,\M)}
\def\bSM{(\overline \S,\overline \M)}

\def\P{{\mathbb{P}}}
\def\Q{{\mathbb{Q}}}
\def\R{{\mathbb{R}}}
\def\Z{{\mathbb{Z}}}
\def\ZP{{\mathbb{ZP}}}

\def\1{\mathbbm{1}}
\def\x{{\mathbf x}}

\def\ens#1{\left\{ #1 \right\}}

\def\Hom{{\rm{Hom}}}

\def\op{{\rm{op}\,}}
\def\min{{\rm{min}\,}}

\def\Iff{\Leftrightarrow}

\newcommand{\rp}{\mathbb{R} {\rm P}^2}
\newcommand{\wt}[1]{\widetilde{#1}}
\newcommand\pgl{{\rm PGL}(2 , \R)}
\newcommand{\tr}{\mathrm{tr}}
\def\Hy{{\mathbf{H}}^2}
\newcommand\psl{{\rm PSL}(2 , \R)}
\newcommand{\SL}{\mathrm{SL}}

\def\crosscap#1#2#3{
	\draw[fill=white] (#1,#2) circle (#3);
	\draw (#1,#2) -- ++(45:#3);
	\draw (#1,#2) -- ++(-45:#3);
	\draw (#1,#2) -- ++(135:#3);
	\draw (#1,#2) -- ++(-135:#3);
}

\def\cc#1#2{
	\crosscap{#1}{#2}{.25}
}

\def\mobdeux#1#2{
	\draw[thick] (#1,#2+2) .. controls (#1-2,#2+1) and (#1-2,#2-1) .. (#1,#2-2);
	\draw[thick] (#1,#2+2) .. controls (#1+2,#2+1) and (#1+2,#2-1) .. (#1,#2-2);

	\fill (#1,#2+2) circle (.1);
	\fill (#1,#2-2) circle (.1);

	\cc {#1} {#2}
}

\def\arca#1#2{
	\draw[red] (#1,#2-2) .. controls (#1-2,#2+.75) and (#1+2,#2+.75) .. (#1,#2-2);
}

\def\arcb#1#2{
	\draw[red] (#1,#2+2) .. controls (#1-2,#2-.75) and (#1+2,#2-.75) .. (#1,#2+2);
}

\def\arcca#1#2{
	\draw[red] (#1,#2-2) .. controls (#1-3,#2+1.5) and (#1+3,#2+1.5) .. (#1,#2-2);
}

\def\arccb#1#2{
	\draw[red] (#1,#2+2) .. controls (#1-3,#2-1.5) and (#1+3,#2-1.5) .. (#1,#2+2);
}

\def\arcc#1#2{
	\draw[red] (#1,#2-2) -- (#1,#2+2);
}

\def\arcd#1#2{
	\draw[red] (#1,#2+.25) .. controls (#1+.75,#2+1) and (#1+.75,#2-1) .. (#1,#2-.25);
}

\def\mobone#1#2{
	\draw[thick] (#1,#2-3) .. controls (#1-6,#2+4) and (#1+6,#2+4) .. (#1,#2-3);
	\fill (#1,#2-3) circle (.1);

	\cc {#1} {#2}
}

\def\arcaone#1#2{
	\draw[red] (#1,#2-3) .. controls (#1-3,#2+1.1) and (#1+3,#2+1.1) .. (#1,#2-3);
}

\def\mobtrois#1#2{
	\draw[thick,fill=white] (#1-2,#2-2) -- (#1+2,#2-2) -- (#1,#2+3) -- cycle;
	\fill (#1,#2+3) circle (.1);
	\fill (#1-2,#2-2) circle (.1);
	\fill (#1+2,#2-2) circle (.1);

}

\def\arcabg#1#2{
	\draw[red] (#1-2,#2-2) .. controls (#1-.25,#2+1.75) and (#1+2,#2) .. (#1-2,#2-2);
}

\def\arccbg#1#2{
	\draw[red] (#1-2,#2-2) .. controls (#1-.5,#2+2.5) and (#1+2.5,#2-.5) .. (#1-2,#2-2);
}

\def\arcabd#1#2{
	\draw[red] (#1+2,#2-2) .. controls (#1+.25,#2+1.75) and (#1-2,#2) .. (#1+2,#2-2);
}

\def\arccbd#1#2{
	\draw[red] (#1+2,#2-2) .. controls (#1+.5,#2+2.5) and (#1-2.5,#2-.5) .. (#1+2,#2-2);
}

\def\arcah#1#2{
	\draw[red] (#1,#2+3) .. controls (#1-1.5,#2-1.25) and (#1+1.5,#2-1.25) .. (#1,#2+3);
}

\def\arcch#1#2{
	\draw[red] (#1,#2+3) .. controls (#1-2,#2-1.75) and (#1+2,#2-1.75) .. (#1,#2+3);
}

\def\arcabgh#1#2{
	\draw[red] (#1-2,#2-2) .. controls (#1-1,#2-1.75) and (#1+.5,#2-1.75) .. (#1,#2+3);
}

\def\arccbgh#1#2{
	\draw[red] (#1-2,#2-2) .. controls (#1-2,#2-1.75) and (#1+2,#2-1.75) .. (#1,#2+3);
}

\def\arcabdh#1#2{
	\draw[red] (#1+2,#2-2) .. controls (#1+1,#2-1.75) and (#1-.5,#2-1.75) .. (#1,#2+3);
}

\def\arccbdh#1#2{
	\draw[red] (#1+2,#2-2) .. controls (#1+2,#2-1.75) and (#1-2,#2-1.75) .. (#1,#2+3);
}

\def\arcab#1#2{
	\draw[red] (#1-2,#2-2) .. controls (#1-.75,#2+.75) and (#1+.75,#2+.75) .. (#1+2,#2-2);
}

\def\arccbb#1#2{
	\draw[red] (#1-2,#2-2) .. controls (#1-.5,#2+1.5) and (#1+.5,#2+1.5) .. (#1+2,#2-2);
}

\begin{document}

\title{Quasi-cluster algebras from non-orientable surfaces} 
\author{Gr\'egoire Dupont}
\address{ESPE de Guadeloupe, Morne Ferret, BP 517, 97178 Abymes CEDEX}
\email{gdupont@espe-guadeloupe.fr}
	
\author{Fr\'ed\'eric Palesi}
\address{Aix Marseille Université, CNRS, Centrale Marseille, I2M, UMR 7373, 13453 Marseille, France}
\email{frederic.palesi@univ-amu.fr}
	
\date{\today}

\thanks{
	This paper was written while both authors were at the Universit\'e de Sherbrooke. The first author was a CRM-ISM postdoctoral fellow under the supervision of Ibrahim Assem, Thomas Br\"ustle and Virginie Charette and was also partially funded by the Tomlinson's Scolarship of Bishop's University. The second author was a CIRGET postdoctoral fellow under the supervision of Virginie Charette and Steve Boyer, and was partially funded by ANR 2011 BS 01 020 01 ModGroup. Both authors would like to thank their respective supervisors for interesting discussions on this topic. 
}

\begin{abstract}
	With any non necessarily orientable unpunctured marked surface $\SM$ we associate a commutative algebra $\mathcal A_{\SM}$, called \emph{quasi-cluster algebra}, equipped with a distinguished set of generators, called \emph{quasi-cluster variables}, in bijection with the set of arcs and one-sided simple closed curves in $\SM$. Quasi-cluster variables are naturally gathered into possibly overlapping sets of fixed cardinality, called \emph{quasi-clusters}, corresponding to maximal non-intersecting families of arcs and one-sided simple closed curves in $\SM$. If the surface $\S$ is orientable, then $\mathcal A_{\SM}$ is the cluster algebra associated with the marked surface $\SM$ in the sense of Fomin, Shapiro and Thurston. 

	We classify quasi-cluster algebras with finitely many quasi-cluster variables and prove that for these quasi-cluster algebras, quasi-cluster monomials form a linear basis.

	Finally, we attach to $\SM$ a family of discrete integrable systems satisfied by quasi-cluster variables associated to arcs in $\mathcal A_{\SM}$ and we prove that solutions of these systems can be expressed in terms of cluster variables of type $A$.
\end{abstract}


\maketitle

\section{Introduction}
	Cluster algebras were initially introduced by Fomin and Zelevinsky in order to study total positivity and dual canonical bases in algebraic groups \cite{cluster1}. Since then, cluster structures have appeared in various areas of mathematics like Lie theory, combinatorics, representation theory, mathematical physics or Teichm\"uller theory. The deepest connections between cluster structures and Teichm\"uller theory is found in the work of Fock and Goncharov \cite{FG:higherTeichmuller}. This latter work led Fomin, Shapiro and Thurston to introduce a particular class of cluster algebras, called \emph{cluster algebras from surfaces} \cite{FST:surfaces}. Such a cluster algebra $\mathcal A_{\SM}$ is associated to a so-called \emph{marked surface} $\SM$, that is a 2-dimensional oriented Riemann surfaces $\S$ with a set $\M$ of marked points. These cluster algebras carry a rich combinatorial structure which was studied in detail, see for instance \cite{MSW:positivity,CCS1,BZ:clustercatsurfaces}. Moreover, it turns out that these combinatorial structures actually reflect geometric properties of the surfaces at the level of the corresponding decorated Teichm\"uller space in the following sense : cluster variables in $\mathcal A_{\SM}$ correspond to $\lambda$-lengths of arcs in $\SM$ and relations between these cluster variables correspond to geometric relations between the corresponding $\lambda$-lengths, see \cite{FT:surfaces2} or \cite[Section 6.2]{GSV:book}. Therefore, the framework of cluster algebras provide a combinatorial framework for studying the Teichm\"uller theory associated to the marked surface $\SM$.

	A key ingredient in the construction of $\mathcal A_{\SM}$ by Fomin, Shapiro and Thurston is the orientability of the surface $\S$. If it is not orientable, then it is in fact not possible to define an exchange matrix and thus an initial seed for the expected cluster algebra. However, relations between $\lambda$-lengths of arcs in $\SM$ can still be described. Using this approach, we associate to any 2-dimensional Riemann marked surface $\SM$, orientable or not, and without punctures, a commutative algebra $\mathcal A_{\SM}$. This algebra is endowed with a distinguished set of generators, called \emph{quasi-cluster variables}, gathered into possibly overlapping sets of fixed cardinality, called \emph{quasi-clusters}, defined by a recursive process called \emph{quasi-mutation}. In this context, the set of quasi-cluster variables is in bijection with the set of arcs and one-sided simple closed curves in $\SM$. The quasi-clusters correspond to maximal collections of arcs and simple one-sided closed curves without intersections, referred to as \emph{quasi-triangulations}, and the notion of quasi-mutation generalises the classical notion of \emph{flip} (sometimes called \emph{Whitehead move}) of a triangulation. As in the orientable case, the algebra $\mathcal A_{\SM}$ imitates the relations for the $\lambda$-lengths of the corresponding curves on the decorated Teichm\"uller space. And if the surface $\SM$ is orientable, then the quasi-cluster algebra $\mathcal A_{\SM}$ coincides with the usual cluster algebra associated to the choice of any orientation of $\SM$.

	We initiate a systematic study of these algebras in the spirit of the study of cluster algebras arising from surfaces. In order to enrich the structure of the quasi-cluster algebra, we first establish numerous identities between the $\lambda$-lengths of curves in any marked surface. In particular, Theorem \ref{theorem:resolution} proves analogues of so-called ``skein relations'' for arbitrary curves in a non-necessarily orientable marked surface, see also \cite{MW:resolutions} for an alternative approach in the orientable case.

	We prove that if $\SM$ is non-orientable, then the structure of $\mathcal A_{\SM}$ can be partially studied through the classical cluster algebra associated to the double cover of $\SM$. However, not all the structure of $\mathcal A_{\SM}$ is encoded in this double cover and $\mathcal A_{\SM}$ provides a new combinatorial setup. We prove in Theorem \ref{thm:laurent} that quasi-cluster algebras satisfy the Laurent property. 

	We prove in Theorem \ref{theorem:classification} that the quasi-cluster algebras with finitely many quasi-cluster variables are those which are associated either with a disc or with a M\"obius strip with marked points of the boundary. In this case, we prove a non-orientable analogue of a classical result of Caldero and Keller \cite{CK1} (see also \cite{MSW:bases}) stating that the set of monomials in quasi-cluster variables belonging all to a same quasi-cluster form a linear basis in a quasi-cluster algebra of finite type (Theorem \ref{theorem:basis}).

	Finally, with any unpunctured marked surface $\SM$, we associate in a uniform way a family of discrete integrable systems satisfied by the quasi-cluster variables corresponding to arcs in $\SM$. This construction does not depend on the orientability of the surface and allows one to realise quasi-cluster variables corresponding to arcs in any quasi-cluster algebra $\mathcal A_{\SM}$ associated to a marked surface $\SM$ as analogues of cluster variables of type $A$.

\section{Preliminaries}
	\subsection{Bordered surfaces with marked points}
		In \cite{FST:surfaces}, Fomin, Shapiro and Thurston defined the notion of a \emph{bordered surface with marked points} $\SM$ where $\S$ is a 2-dimensional Riemann surface with boundary. Implicitly in their definition, the surface $\S$ is orientable. We extend the definition to include non-orientable surfaces as well. 

		Recall that a closed (without boundary or puncture) non-orientable surface is homeomorphic to a connected sum of $k$ projective planes $\rp$. The number $k$ is called the \emph{non-orientable genus} of the surface or simply the genus when no confusion arises. A classical result states that the connected sum of a closed non-orientable surface of genus $k$ with a closed orientable surface of genus $g$ is homeomorphic to a closed non-orientable surface of genus $2g+k$, see \cite{Massey:introduction}. The Euler characteristic of a non-orientable surface $\S$ of genus $k$ is given by $\chi (\S) = 2- k$.

		Let $\S$ be a 2-dimensional manifold with boundary $\partial \S$. Fix a non-empty set $\M$ of marked points in the closure of $\S$, so that there is at least one marked point on each connected component of $\partial \S$. Marked points in the interior of $\S$ are called {\it punctures}.

		Up to homeomorphism, $\SM$ is defined by the following data~:
		\begin{itemize}
			\item the orientability of the manifold $\S$;
			\item the genus $g$ of the manifold;
			\item the number $n$ of boundary components;
			\item the integer partition $(b_1 , \dots , b_n)$ corresponding to the number of marked points on each boundary component;
			\item the number $p$ of punctures.
		\end{itemize}

		For the sake of clarity, in the rest of this article, we will only deal with unpunctured surfaces, namely $p=0$. We also want to exclude trivial cases where $\SM$ does not admit any triangulation by a non-empty set of arcs with endpoints at $\M$, consequently we do not allow $\SM$ to be a an unpunctured monogon, digon or triangle. 
		
	\subsection{Quasi-arcs}
		In non-orientable surfaces, the closed curves are classified into two disjoint sets which will play an important role in this article.

		\begin{defi}
			A closed curve on $\S$ is said to be \emph{two-sided} if it admits a regular neighborhood which is orientable. Else it is said to be \emph{one-sided}. 
		\end{defi}

		Any one-sided curve will reverse the local orientation. Hence a surface contains a one-sided curve if and only if the surface is non-orientable. In the orientable case, we do not worry about such curves.

		An \emph{arc} is the isotopy class of a simple curve in $\SM$ joining two marked points. We denote by $\A\SM$ the set of arcs in $\SM$. A \emph{quasi-arc} in $\SM$ is either an arc or a simple one-sided closed curve in the interior of $\S$. We denote by $\qA\SM$ the set of quasi-arcs in $\SM$. Note that if $\SM$ is orientable, then $\qA\SM= \A\SM$. We denote by $\B\SM$ the set of connected components of $\d\S \setminus \M$, which we call \emph{boundary segments}. 

		To draw non-orientable surfaces, we use the identification of $\rp$, as the quotient of the unit sphere $\mathbb{S}^2 \subset \R^3$ by the antipodal map. When cutting the sphere along the equator, we see that the projective plane is homeomorphic to a closed disc with opposite points on the boundary identified, which is called a \emph{crosscap}. Hence a closed non-orientable surface of genus $k$ is identified with a sphere where $k$ open discs have been removed and the opposite points of each boundary components identified. A crosscap is represented as a circle with a cross inside, see Figure \ref{fig:conventions}.

		\begin{figure}
			\begin{center}
				\begin{tikzpicture}[scale = .75]
					\fill[red] (.5,-.5) node {$d$};
					\arcd 0 0
					\mobdeux 0 0

					\fill[red] (6.75,.75) node {$e$};
					\draw[red] (6,0) circle (.75);
					\mobdeux 6 0

					\fill[red] (12.5,0) node {$c$};
					\arcc {12} 0
					\mobdeux {12} 0
				\end{tikzpicture}
			\end{center}
			\caption{Conventions of drawings in the M\"obius strip with two marked points~: $d$ is a one-sided closed curve, $e$ is a two-sided closed curve and $c$ is an arc.}\label{fig:conventions}
		\end{figure}
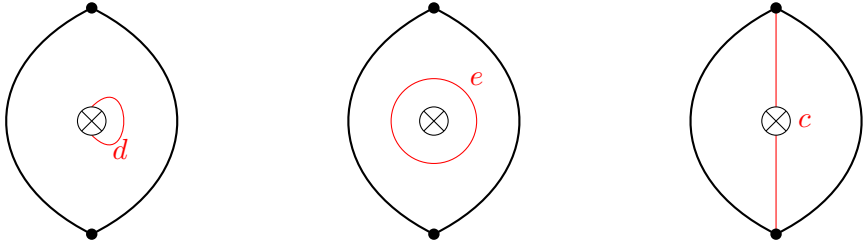

	\subsection{Decorated Teichm\"uller space}
		The classical definitions of Teichm\"uller spaces and decorated Teichm\"uller spaces can easily be extended to include non-orientable surfaces as well, see \cite{Penner:lambda} for a complete exposure.  

		\begin{defi}
			The \emph{Teichm\"uller space} $\mathcal{T} \SM$ consists of all complete finite-area hyperbolic structures with constant curvature $-1$ on $\S \setminus \M$, with geodesic boundary at $\partial \S \setminus \M$. 
		\end{defi}

		\begin{defi}
			A point of the \emph{decorated Teichm\"uller space} $\wt{\mathcal{T}} \SM$ is  a hyperbolic structure as above together with a collection of horocycles, one around each marked point. 
		\end{defi}

		Fix a hyperbolic structure on $\S$, namely an element in $\mathcal{T} \SM$. For any curve $c$ joining two punctures, there is a unique geodesic in its homotopy class. We call this element the geodesic representative of $c$, and by a slight abuse of notation, we will also denote it $c$. Likewise, every closed curve can be represented by a unique geodesic representative on the surface. Recall that any element of the fundamental group $\pi_1 (\S)$ gives rise to the homotopy class of a closed curve, and hence a geodesic representative.

		Given a decorated hyperbolic structure on $\SM$, we recall the definition of Penner's $\lambda$-lengths of a decorated ideal arc and extend it to closed curves.

		\begin{defi} Let $\wt{\sigma} \in \wt{\mathcal{T}} \SM$ be a decorated hyperbolic structure.
		\begin{itemize}
		\item Let $a$ be a decorated ideal arc in $\A \SM$ or in $\B \SM$. The $\lambda$-length of $a$ is defined as
		$$\lambda_{\wt{\sigma}} (a) = \exp \left(\frac{l(a)}{2}\right)$$
		where $l(a)$ is the signed hyperbolic distance along $a$ between the two horocycles at either end of $a$.
		\item Let $b$ be a two-sided closed curve.  The $\lambda$-length of $b$ is defined as 
		$$\lambda_{\wt{\sigma}} (b) = \exp \left( \frac{l(b)}{2} \right) + \exp \left(\frac{-l(b)}{2}\right) = 2 \cosh \left(\frac{l(b)}{2}\right) $$
		where $l(b)$ is the hyperbolic length of the geodesic representative of $b$.
		\item Let $d$ be a one-sided closed curve.  The $\lambda$-length of $d$ is defined as 
		$$\lambda_{\wt{\sigma}} (d) = \exp \left( \frac{l(d)}{2} \right) - \exp \left(\frac{-l(d)}{2}\right) = 2 \sinh \left(\frac{l(d)}{2}\right) $$
		where $l(d)$ is the hyperbolic length of the geodesic representative of $d$.
		\end{itemize}
		\end{defi}

		Remark that this definition does not need the arcs or curves to be simple. In fact we can extend this definition to a finite union of arcs and closed curves.

		\begin{defi} A \emph{multigeodesic} $\alpha$ is a multiset based on the set $ \{ a_1,   \dots , a_n\}$ where each $a_i$ is a curve joining two marked points or any closed curve. Each element of the set has a multiplicity $m_i$. The $\lambda$-length of such a multigeodesic is given by :
		$$\lambda_{\wt{\sigma}} (\alpha) = \prod_{i=1}^n \left(\lambda_{\wt{\sigma}} (a_i)\right)^{m_i}.$$
		\end{defi}

		For a given multigeodesic $\alpha$, one can view the $\lambda$-length as a positive function on the decorated Teichm\"uller space $\wt{\mathcal{T}} \SM$ in the following sense~:
		$$\lambda (\alpha) :
		\left\{\begin{array}{rcl}
			\wt{\mathcal{T}} \SM & \longrightarrow & \R_{>0} \\
			\wt{\sigma} & \longmapsto & \lambda_{\wt{\sigma}} (\alpha).
		\end{array}\right.$$

		Let $\wt{\sigma} \in \wt{\mathcal{T}} \SM$. The holonomy map $\rho_{\sigma}$ of the underlying hyperbolic structure $\sigma \in \mathcal{T} \SM$ defines a homeomorphism from $ \mathcal{T} \SM $ to a connected component of the moduli space 
		$$\Hom ( \pi_1 (\S) , G ) / G $$
		where $G$ is the group  $ \pgl$ of isometries of the hyperbolic plane. The set $\Hom ( \pi_1 (\S) , G )$ is the set of morphism $\rho : \pi_1 (\S) \rightarrow G$ and the $G$-action is by conjugation. Hence any decorated hyperbolic structure $\sigma \in  \mathcal{T} \SM$ gives rise to a conjugacy class of representations  $[ \rho_\sigma ] : \pi_1 (\S) \rightarrow G$.

	For any element of $\pgl$, the absolute value of the trace is well-defined. Let $b$ be a closed curve (one or two-sided) corresponding to an element $\mathbf{b} \in \pi_1 \SM$. Let $\sigma \in  \mathcal{T} \SM$ be a hyperbolic structure and $\rho_\sigma$ be a representative of the conjugacy class $[ \rho_\sigma ]$. The trace is invariant under conjugation, and hence the value of $ | \tr (\rho_\sigma (\mathbf{b}) ) | $ is well-defined and does not depend on the choice of $\rho_\sigma$. Moreover, $\rho_\sigma (\mathbf{b})$ is hyperbolic and a classical result in hyperbolic geometry states that :
	$$\lambda_{\wt{\sigma}} (b) = | \tr (\rho_\sigma (\mathbf{b}) ) |.$$

\section{Quasi-cluster complexes associated with non-orientable surfaces}
	Let $\SM$ be a bordered marked surface without punctures orientable or not. Two elements in $\qA\SM$ are called \emph{compatible} if they are distinct and do not intersect each other. 

	\begin{defi}
		A \emph{quasi-triangulation} of $\SM$ is a maximal collection of compatible elements in $\qA\SM$. A quasi-triangulation is called a \emph{triangulation} if it consists only of elements in $\A\SM$.
	\end{defi}

	\begin{prop}
		Let $T \in \T^\otimes \SM$ be a quasi-triangulation. Then $T$ cuts $\S$ into a finite union of triangles and annuli with one marked point. The number of annuli is the number of one-sided curves in the quasi-triangulation $T$. 
	\end{prop}
	\begin{proof}
		Cut the surface $\S$ open along all arcs and curves of $T$. This splits the surface into a finite union of connected components.  Let $K$ be one of these components. Then $K$ is bordered by at least one boundary component which has at least one marked point. As $T$ is a maximal set of arcs and curves, $K$ does not have any interior quasi-arc. 

		First, we notice that $K$ cannot be non-orientable. Indeed, non-orientability would imply that there exists a one-sided simple closed curve in $K$, which would be a non-trivial interior quasi-arc.

		Assume that $K$ has only one boundary component $\partial K$. Let $m$ be the number of marked points on $\partial K$. If $m = 1$ then the boundary arc is trivial which is excluded. If $m = 2$ then the two boundary arcs are homotopic which is excluded. And if $m \geq 4$, then $K$ would admit interior non-trivial arcs as diagonal of the $m$-gon. Hence, we infer that $m = 3$ and $K$ is a triangle.

		Then suppose that $K$ has two boundary components. If both components have marked points, then the curve joining the marked points of each boundary components would be a non-trivial interior arc of $K$. Hence, necessarily one of the boundary is unmarked. Moreover, if the marked boundary component has more than one marked point then the non trivial curve joining one marked point to itself going around the unmarked boundary component will not be homotopic to a boundary segment of $K$ and hence will be a non-trivial interior quasi-arc.

		Finally, $K$ cannot have three or more boundary components, as an arc from the marked point to itself going around one unmarked boundary but not the other one would be a non-trivial interior arc. So $K$ is either a triangle or an annuli with one marked point which proves the first part of the proposition.

		For the second part of the proposition, we simply notice that an unmarked boundary component can only be obtained by cutting along a simple closed curve. Hence, the number of annuli is exactly the number of one-sided curves on the surface $\S$.
	\end{proof}

	\subsection{Quasi-mutations}
		\begin{defi}
			An \emph{anti-self-folded triangle} is any triangle of a quasi-triangulation with two edges identified by an orientation-reversing isometry.
		\end{defi}

		\begin{prop}\label{prop:flip}
			Let $\SM$ be an unpunctured marked surface and let $T$ be a quasi-triangulation of $\SM$. Then for any $t \in T$, there exists a unique $t' \in \qA\SM$ such that $t' \neq t$ and such that $\mu_t(T) = T \setminus \ens{t} \sqcup \ens{t'}$ is a quasi-triangulation of $\SM$. 
		\end{prop}
		\begin{proof}
			If $t$ is an arc separating two different triangles, then this is standard : the two triangles define a quadrilateral with $t$ as a diagonal and $t'$ is the unique other diagonal.
			
			If $t$ is an arc which is an edge of a single triangle $\Delta$, then either $\Delta$ is a self-folded triangle or an anti-self-folded triangle. As we have excluded punctured surfaces, $\Delta$ is necessarily an anti-self-folded triangle. Denote the third side of $\Delta$ by $c$. Then $c$ is an arc bounding a M\"obius strip $N$ and $t$ is the only non-trivial arc in $N$. There is a unique non-trivial simple closed curve $t'$ in $N$ corresponding to the core of the M\"obius strip. The curve $t'$ and the arc $t$ intersect once, and hence $t'$ is the desired element of $\qA\SM$.
			
			Similarly, if $t$ is a one-sided simple closed curve, then $t$ lies inside a M\"obius strip $N$ bounded by an arc $c$. And $t'$ is the only non-trivial arc inside $N$.
			
			If $t$ is an arc separating a triangle from an annuli, then we are in the situation given in Figure \ref{fig:qmutM2bis}. The mutation is exactly a quasi-flip in the sense of Penner (see \cite{Penner:bordered}) which gives the unicity of the arc $t'$.
			
			Finally, if $t$ is an arc separating two annuli, then necessarily $\S$ is a once-punctured Klein bottle which we have excluded from our hypotheses on $\SM$.		
		\end{proof}

		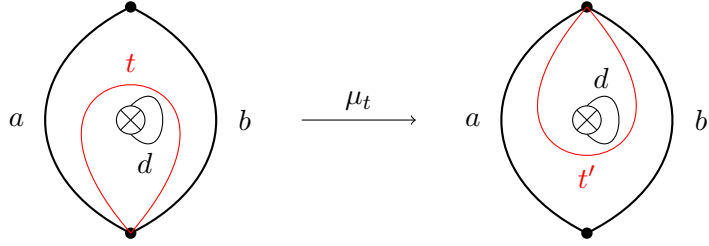
\begin{figure}
			\def\arcdnoir#1#2{
					\draw (#1,#2+.25) .. controls (#1+.75,#2+1) and (#1+.75,#2-1) .. (#1,#2-.25);
				}

			\begin{center}
				\begin{tikzpicture}[scale = .75]

					\arcdnoir {0} 0
					\mobdeux 0 0
					\arcca 0 0
					\fill[red] (0,1) node [] {$t$};

					\fill (-2,0) node {$a$};
					\fill (2,0) node {$b$};

					\fill (.25,-.75) node {$d$};
				

					\draw[->] (3,0) -- (5,0);
					\fill (4,0) node [above] {$\mu_t$};

					
					\arcdnoir {8} 0
					\mobdeux {8} 0
					\arccb {8} 0
					\fill[red] (8,-1) node [] {$t'$};

					\fill (8-2,0) node {$a$};
					\fill (8+2,0) node {$b$};

					\fill (8.25,.75) node {$d$};

				\end{tikzpicture}
			\end{center}
			\caption{A quasi-mutation}\label{fig:qmutM2bis}
		\end{figure}

		\begin{defi}
			With the notation of Proposition \ref{prop:flip}, the quasi-triangulation $\mu_t(T)$ is called the \emph{quasi-mutation of $T$ in the direction $t$} and the element $t'$ in $\qA\SM$ is called the \emph{quasi-flip} of $t$ with respect to $T$.

			If both $t$ and $t'$ are arcs, then $\mu_t$ is called a \emph{mutation} and $t'$ is called the \emph{flip} of $t$ with respect to $T$.
		\end{defi}

		\begin{exmp}
			Figure \ref{fig:qmutM2} depicts examples of two quasi-mutations in the M\"obius strip $\mathcal M_2$ with two marked points. The quasi-mutation $\mu_{c_b}$ is a mutation whereas the quasi-mutation $\mu_b$ is not.
			\begin{figure}
				\begin{center}
					\begin{tikzpicture}[scale = .75]
						\arcb 0 0
						\arccb 0 0 
						\mobdeux 0 0
						\fill[red] (0,.5) node [right] {$b$};
						\fill[red] (1.1,0) node {$c_b$};

						\draw[->] (2,-1) -- (4,-2);
						\fill (3,-1) node [above,right] {$\mu_b$};

						\draw[->] (-2,-1) -- (-4,-2);
						\fill (-3,-1) node [above,left] {$\mu_{c_b}$};

						\arcb {-6} {-4}
						\arcc {-6} {-4}
						\mobdeux {-6} {-4}
						\fill[red] (-6,-3.5) node [right] {$b$};
						\fill[red] (-6,-5) node [right] {$c$};

						\arcd {6} {-4}
						\arccb {6} {-4}
						\mobdeux {6} {-4}
						\fill[red] (6,-3.2) node [right] {$d$};
						\fill[red] (7.1,-4) node {$c_b$};
					\end{tikzpicture}
				\end{center}
				\caption{Examples of quasi-mutations in $\mathcal M_2$.}\label{fig:qmutM2}
			\end{figure}
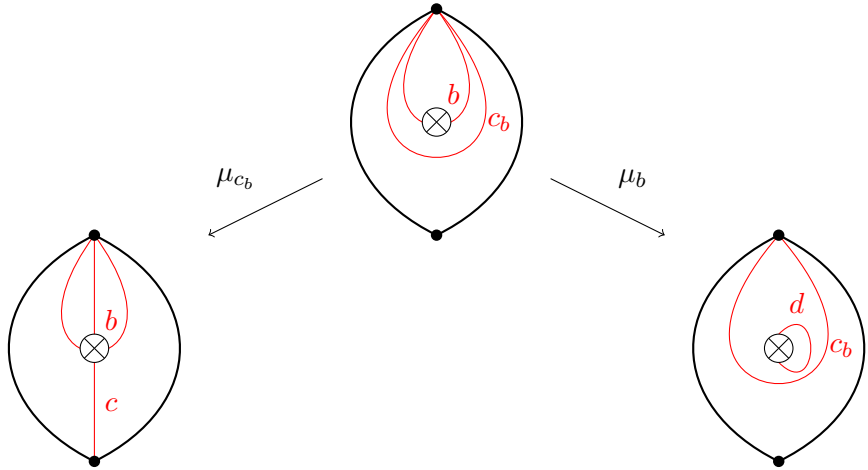
		\end{exmp}

		\begin{prop}\label{prop:rank}
			Let $\SM$ be a marked surface without puncture. Then the number of elements in a quasi-triangulation does not depend on the choice of the quasi-triangulation and is called the \emph{rank} of the surface $\SM$.
		\end{prop}
		\begin{proof}
			For triangulations, this is easily done by consideration on the Euler characteristic of the surface (see for instance \cite{FG:dualTeichmuller}) and the number of interior arcs for a non-orientable surface of genus $k$ with $p$ punctures and $n$ boundary components having each $b_i$ marked points is given by
			$$ N = 3k - 6 + 3n + 3p + \sum_{i=1}^n b_i .$$
			For quasi-triangulation containing one or more one-sided simple closed curves, we can associate to each one-sided curve the unique arc given by the Proposition \ref{prop:flip}. Therefore to each quasi-triangulation corresponds a triangulation which has the same number of elements.
		\end{proof}

		\begin{exmp}
			For any $n \geq 1$, we denote by $\mathcal M_n$ the M\"obius strip with $n$ marked points on the boundary. It is a non-orientable surface of rank $n$.
		\end{exmp}

		\begin{corol}\label{corol:caracqmut}
			Let $\SM$ be an unpunctured marked surface and let $T$ be a triangulation of $\SM$. Then for any $t \in T$, the quasi-mutation in the direction $t$ of $T$ is a mutation if and only if $t$ is not the internal arc of an anti-self-folded triangle in $T$.

			In this case, we say that $t$ is \emph{mutable with respect to $T$}.
		\end{corol}
		\begin{proof}
			If $t$ is not the internal arc of an anti-self-folded triangle in $T$, then removing $t$ in $T$ delimits a quadrilateral $Q$ in which $t$ is a diagonal and $t'$ is the other diagonal. In particular, $t'$ is an arc and $\mu_t$ is a mutation.

			Conversely, if $t$ is the internal arc of an anti-self-folded triangle, then there is an arc $x$ in $T$ such that locally around $t$, the triangulation looks like the situation depicted in Figure \ref{fig:nonmutable}. Thus the quasi-flip $t'$ of $t$ is an element in $\qA\SM \setminus \A\SM$ and $\mu_t$ is not a mutation.
			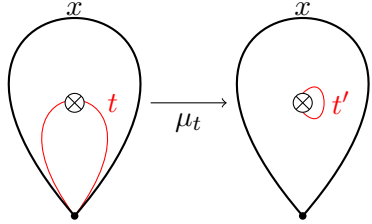
\begin{figure}
				\begin{center}
					\begin{tikzpicture}[scale = .5]
						\arcaone 0 0
						\mobone {0} 0 
						\fill (0,2.5) node {$x$};

						\fill[red] (1,0) node {$t$};

						\draw[->] (2,0) -- (4,0);
						\fill (3,0) node [below] {$\mu_t$};

						\arcd 6 0
						\mobone {6} 0 
						\fill[red] (7,0) node {$t'$};
						\fill (6,2.5) node {$x$};
					\end{tikzpicture}
				\end{center}
				\caption{Quasi-mutation at a non-mutable arc in a triangulation.}\label{fig:nonmutable}
			\end{figure}
		\end{proof}

	\subsection{Quasi-exchange graph}
		Let $\SM$ be an unpunctured marked surface of rank $n \geq 1$. 

		\begin{defi}
			The \emph{quasi-cluster complex} $\Delta^{\otimes}\SM$ is the (possibly infinite) simplicial complex on the ground set $\qA\SM$ defined as the clique complex for the compatibility relation. The vertices in $\Delta^{\otimes}\SM$ are the elements in $\qA\SM$ and the maximal simplices are the quasi-triangulations.

			Similarly, the \emph{cluster complex} $\Delta\SM$ is the simplicial complex on the ground set $\A\SM$ defined as the clique complex for the compatibility relation. In other words, the vertices in $\Delta\SM$ are the elements in $\A\SM$ and the maximal simplices are the triangulations.
		\end{defi}

		\begin{defi}
			The dual graph of $\Delta^{\otimes}\SM$ is denoted by $\qE\SM$ and is called the \emph{quasi-exchange graph} of $\SM$. Its vertices are the quasi-triangulations of $\SM$ and its edges correspond to quasi-mutations.

			The dual graph of $\Delta\SM$ is denoted by $\E \SM$ and is called the \emph{exchange graph} of $\SM$. Its vertices are the triangulations of $\SM$ and its edges correspond to mutations.
		\end{defi}

		\begin{prop}
			$\qE\SM$ is a connected $n$-regular graph.
		\end{prop}
		\begin{proof}
			According to Proposition \ref{prop:flip}, every element in a quasi-triangulation can be quasi-mutated and quasi-mutations in distinct directions give rise to distinct quasi-triangulations. It thus follows that $\qE\SM$ is $n$-regular. Proving that $\qE\SM$ is connected is equivalent to proving that two quasi-triangulations are connected by a sequence of quasi-mutations. The proof that $\E \SM$ is connected can be found in \cite{Hatcher:triangulations} (Hatcher's arguments do not use the orientability of the surface). Now it is enough to observe that each quasi-triangulation  $T$ which is not a triangulation can be related to a triangulation by a sequence of quasi-mutations, one at each one-sided curve in $T$. Therefore, any two quasi-triangulations are related by a sequence of quasi-mutations, which proves the proposition.
		\end{proof}

		\begin{exmp}\label{exmp:qexchM2}
			The cluster and quasi-cluster complexes for the M\"obius strip $\mathcal M_2$ are depicted in Figures \ref{fig:complexM2} and \ref{fig:qcomplexM2}.
			\begin{figure}
				\begin{center}
					\begin{tikzpicture}[scale = .5]
				%
				%
						\arca {-6} 0 
						\mobdeux {-6} 0 

						\arcc 0 0 
						\mobdeux 0 0 

						\arcb {6} 0 
						\mobdeux {6} 0 

						\arcca {-6} {-8} 
						\mobdeux {-6} {-8}


						\arccb {6} {-8} 
						\mobdeux {6} {-8}

						\draw[thick] (-2,0) -- (-4,0);
						\draw[thick] (2,0) -- (4,0);
						\draw[thick] (-6,-3) -- (-6,-5);
						\draw[thick] (6,-3) -- (6,-5);
					\end{tikzpicture}
				\end{center}
				\caption{The cluster complex of the M\"obius strip with two marked points.}\label{fig:complexM2}
			\end{figure}
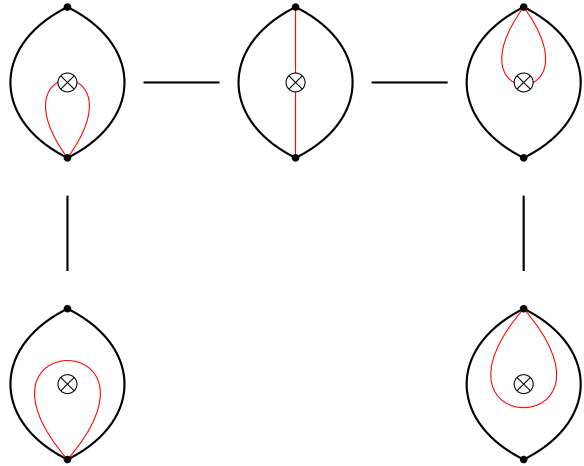

			\begin{figure}
				\begin{center}
					\begin{tikzpicture}[scale = .5]
				%
				%
						\arca {-6} 0 
						\mobdeux {-6} 0 

						\arcc 0 0 
						\mobdeux 0 0 

						\arcb {6} 0 
						\mobdeux {6} 0 

						\arcca {-6} {-8} 
						\mobdeux {-6} {-8}

						\arcd {0} {-8}
						\mobdeux {0} {-8}

						\arccb {6} {-8} 
						\mobdeux {6} {-8}

						\draw[thick] (-2,0) -- (-4,0);
						\draw[thick] (2,0) -- (4,0);
						\draw[thick] (-2,-8) -- (-4,-8);
						\draw[thick] (2,-8) -- (4,-8);
						\draw[thick] (-6,-3) -- (-6,-5);
						\draw[thick] (6,-3) -- (6,-5);
					\end{tikzpicture}
				\end{center}
				\caption{The quasi-cluster complex of the M\"obius strip with two marked points.}\label{fig:qcomplexM2}
			\end{figure}
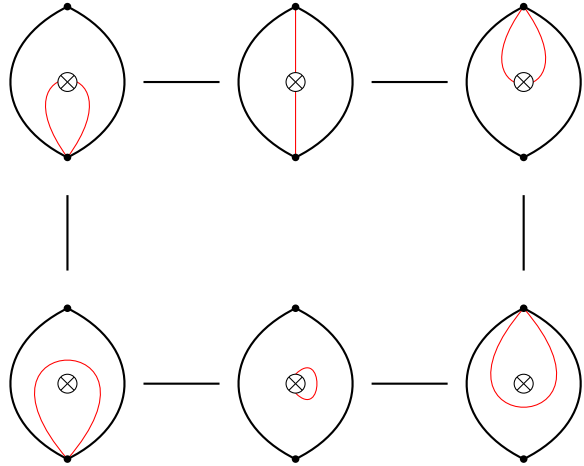
		\end{exmp}

		\begin{exmp}\label{exmp:qexchM3}
			The cluster complex $\Delta(\mathcal M_3)$ of the M\"obius strip with three marked points is depicted in Figure \ref{fig:complexM3}. The exchange graph $\E(\mathcal M_3)$ has 16 vertices and its faces are six pentagons. The quasi-cluster complex $\qE(\mathcal M_3)$ is obtained from the cluster complex by adding the unique one-sided curve in $\mathcal M_3$ as a vertex of the complex and by connecting the six ``external'' vertices of the cluster complex to this unique one-sided curve. Therefore, the quasi-exchange graph $\Delta^{\otimes}(\mathcal M_3)$ is a polytope with 22 vertices and whose faces are three squares, six pentagons and four hexagons.

			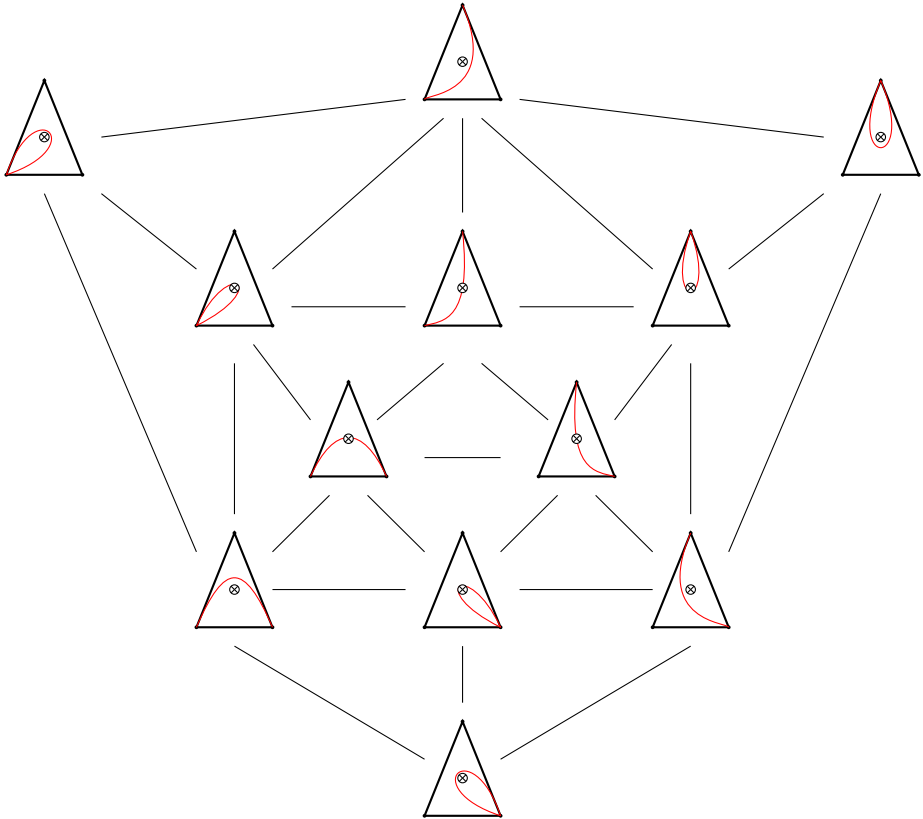
\begin{figure}
				\begin{center}
					\begin{tikzpicture}[scale = .25]


						\draw (8,1) -- (11,5);
						\draw (3,7) -- (9,7);

						\draw (-14,9) -- (-19,13);

						\draw (0,17) -- (0,12);

						\draw (-1,17) -- (-10,9);
						\draw (1,17) -- (10,9);

						\draw (3,18) -- (19,16);
						\draw (-3,18) -- (-19,16);

						\draw (-4.5,1) -- (-1,4);
						\draw (4.5,1) -- (1,4);

						\draw (-2,-1) -- (2,-1);
						
						\draw (-5,-3) -- (-2,-6);
						\draw (5,-3) -- (2,-6);

						\draw (-12,4) -- (-12,-4);
						
						\draw (12,4) -- (12,-4);

						\draw (-14,-6) -- (-22,13);

						\draw (14,-6) -- (22,13);

						\draw (-8,1) -- (-11,5);

						\draw (-3,7) -- (-9,7);

						\draw (14,9) -- (19,13);

						\draw (7,-3) -- (10,-6);

						\draw (-7,-3) -- (-10,-6);

						\draw (-3,-8) -- (-10,-8);

						\draw (3,-8) -- (10,-8);

						\draw (0,-11) -- (0,-14);

						\draw (12,-11) -- (2,-17);

						\draw (-12,-11) -- (-2,-17);

					
						\mobtrois {0} {20}
						\arccbgh {0} {20} 	
						\cc {0} {20} 	
					
						\mobtrois {-22} {16}
						\arccbg {-22} {16} 	
						\cc {-22} {16} 	

						\mobtrois {22} {16}
						\arcch {22} {16} 
						\cc {22} {16}

						\mobtrois {-12} {8}
						\arcabg {-12} {8} 
						\cc {-12} 8

						\mobtrois {12} {8}
						\arcah {12} {8} 
						\cc {12} 8
				
						\mobtrois 0 8
						\arcabgh 0 8 		
						\cc 0 8

						\mobtrois 6 0
						\arcabdh 6 0 			
						\cc 6 0

						\mobtrois {-6} 0
						\arcab {-6} 0 			
						\cc {-6} 0

						\mobtrois {0} {-8}
						\arcabd {0} {-8} 			
						\cc 0 {-8}

						\mobtrois {-12} {-8}
						\arccbb {-12} {-8} 			
						\cc {-12} {-8}
						
						\mobtrois {12} {-8}
						\arccbdh {12} {-8} 
						\cc {12} {-8} 

						\mobtrois {0} {-18}
						\arccbd {0} {-18}
						\cc {0} {-18} 			
					\end{tikzpicture}
				\end{center}
				\caption{The cluster complex of the M\"obius strip $\mathcal M_3$.}\label{fig:complexM3}
			\end{figure}
		\end{exmp}

		\begin{rmq}
			Note that if $\SM$ is not orientable, then $\E\SM$ is not regular, as it appears for instance in Figure \ref{fig:complexM2}.
		\end{rmq}

\section{Relations between quasi-arcs}

	\subsection{Hyperbolic geometry in the upper half-plane}
		We use throughout this paper the upper half-plane model of the hyperbolic plane
		$$\Hy = \left\{ z \in \mathbb{C} | \rm{Im} (z) >0 \right\} $$
		endowed with the Riemannian metric 
		$$ds^2 := \dfrac{dx^2 + dy^2}{y^2}.$$
		Geodesics in $\Hy$ are given either by circles perpendicular to the real axis or by lines parallel to the imaginary axis. The points of the boundary $\partial \Hy$ are elements of $\R \cup \{\infty\}$.

		The group $\pgl$ can be defined as the quotient of two-by-two matrices with determinant plus or minus one, by the group $\{ \pm I \}$. Note that the sign of the determinant is still well-defined on the quotient. It acts on $\Hy$ by M\"obius and anti-M\"obius transformations. The group of isometries of the hyperbolic plane is naturally identified with $\pgl$.  An element with determinant one will correspond to an orientation-preserving isometry, and an element with determinant minus one will correspond to an orientation-reversing isometry.

		An \emph{horocycle} in the upper half-plane is a euclidean circle parallel to the real axis, or a horizontal line parallel to the real axis. Hence a horocycle $U$ is defined by its center $u \in \R \cup \{\infty\}$ and its diameter $h \in \R_{>0}$ (for a horocycle centered at $\infty$ its diameter is the inverse of the height of the parallel), and is denoted $U = (u,h)$.

		A \emph{decorated geodesic} is a geodesic joining two points $u$ and $v$ on $\R \cup \{\infty\}$ together with horocycles $U$ and $V$ centered at $u$ and $v$, and is denoted by $(U,V)$.  For horocycles $U = (u,h)$ and $V = (v,k)$ with distinct centers $u, v \in \R$, one can express the $\lambda$-length of the decorated geodesic as
		$$\lambda (U,V) = \dfrac{|v-u|}{\sqrt{hk}}.$$
		As shown by Penner \cite{Penner:lambda}, we have $\lambda (U , V) = \exp (\delta/2)$ where $\delta$ is the signed hyperbolic distance between the two horocycles along the geodesic.

	\subsection{Decorated Teichm\"uller space}

	The main purpose of $\lambda$-lengths is to provide coordinates on the decorated Teichm\"uller space of an orientable surface, see \cite{Penner:lambda}. We extend this result to include non-orientable surfaces as well using quasi-arcs and quasi-triangulations. First we have to settle the case of a M\"obius strip with one marked point on the boundary in the following proposition~:

		\begin{prop}\label{prop:anti-self}
			Let $c , d \in \R_{>0}$. There exists a unique isometry class of triple of horocycles $(U,V,W)$ such that :
			\begin{itemize}
			\item $\lambda (U , V) = c$ ;
			\item there is an orientation-reversing isometry $D$ such that $D(U) = W$, $D(W)=V$ and $|\tr (D)| = d$.
			\end{itemize}
		\end{prop}
		\begin{proof}
			Let $(U,V,W)$ be a triple of horocycles. If there is an isometry $\phi$ such that $\phi(U) = W$ and $\phi(W) = V$ then we have that $\lambda(U,W) = \lambda (W,V)$.

			For any $a \in \R_{>0}$, there exists a unique isometry class of horocycles $(U,V,W)$ such that $\lambda(U,V) = c$ and $\lambda(U,W) = \lambda (W,U) = a$. Now let $D$ be the unique orientation reversing isometry such that $D(U) = W$ and $D(W) = V$. Up to conjugacy and rescaling we can assume that $D$ is represented by a matrix of the form
			$$D = \left( \begin{array}{cc} \mu & 0 \\ 0 & - 1/\mu \end{array} \right), \mbox{ with } \mu > 1.$$
			If we denote the three horocycles by $U = (u,h)$, $V = (v,k)$ and $W = (w,l)$ then we have the following relations :
			$$w = - \mu^2 u,  \hspace{0.5cm} v = - \mu^2 w = \mu^4 u,  \hspace{0.5cm} k = \mu^2 l = \mu^4 h,$$
			 $$\dfrac{|v-u|}{\sqrt{hk}} = c, \hspace{0.5cm} \dfrac{|w-u|}{\sqrt{hl}} = \dfrac{|v-w|}{\sqrt{kl}} = a.$$
			This gives :
			$$c = \dfrac{|u|}{h} \left( \mu^2 - \dfrac{1}{\mu^2} \right) \mbox{  and  } a = \dfrac{|u|}{h} \left( \mu + \dfrac{1}{\mu} \right).$$
			We infer that
			$$|\tr (D) | = \mu - \dfrac{1}{\mu} = \dfrac{c}{a}.$$

			We conclude that for any $c,d>0$, there exists a unique isometry class of triple of horocycles with $\lambda$-length $(c , c/d , c/d)$. This isometry class satisfies the property that an orientation-reversing isometry sending one of the side of length $c/d$ on the other one, has a trace of absolute value $d$. 
		\end{proof}

		\begin{theorem}\label{theorem:lambdaindalg}
			For any quasi-triangulation $T \in \qT\SM$, the natural mapping 
			$$\Lambda_T : \left\{
			\begin{array}{rcl}
				\wt{\mathcal{T}} (\S , M) & \longrightarrow & \R_{> 0}^{T \cup \B\SM} \\
				\sigma & \longmapsto & ( t \mapsto \lambda_\sigma (t))
			\end{array}\right.$$
			is a homeomorphism.
		\end{theorem} 
		\begin{proof}
			For an orientable surface with boundaries, this is the classical result of Penner on coordinates for the decorated Teichm\"uller space \cite{Penner:bordered}.

			If $\SM$ is a non-orientable surface and $T$ is a triangulation (without one-sided closed curves), then this theorem is a straightforward generalisation of Penner's result. We give here the argument that differs and we refer to \cite{Penner:lambda} for the sake of completeness. 

			Recall that the idea of the original proof is to produce an inverse for the map $\Lambda_T$. So suppose there is a positive real number assigned to each arc in a triangulation $T$. From the triangulation of the surface $\S$, we get a triangulation of the universal cover $\wt{\S}$. From this, we get a corresponding triangulation of the hyperbolic plane $\Hy$, constructed by induction on the set of triangles. This gives a homeomorphism $\phi : \wt{\S} \rightarrow \Hy$, which is the developing map for the hyperbolic structure. 

			The holonomy map $\rho : \pi_1 (\S) \rightarrow \pgl$ defined by the developing map $\phi$ sends one-sided curves to orientation-reversing isometries. These isometries are elements of $\pgl$ that are not in $\psl$. The group $\pgl$ acts transitively on triples of horocycles, whereas $\psl$ acts transitively only on positively oriented triples of horocycles. So we can get anti-M\"obius transformations in addition to M\"obius transformations between two identified triangles in $\Hy$. This is the only slight difference with the orientable case and this does not change the other arguments of Penner's proof.

			The only thing that is left to show, is that the theorem still holds for quasi-triangulations containing one-sided simple closed curves. For any one-sided closed curve in a quasi-triangulation, we have a unique corresponding arc that bounds a M\"obius strip. Suppose there is only one such curve $d$ and cut the surface along the corresponding arc $c$. We get a subsurface $\S'$ with $c$ as a boundary arc, and the surface $\S$ is obtained by gluing a M\"obius strip along $c$. The quasi-triangulation of $\S$ restricted to $\S'$ is a triangulation and we can apply the preceding arguments to construct the unique hyperbolic structure on $\S'$ defined by the $\lambda$-lengths. 

			Then we use the Proposition \ref{prop:anti-self} to show that the $\lambda$-length of the one-sided curve $d$ together with the $\lambda$-length of $c$ uniquely define a hyperbolic structure on the M\"obius strip. There is no restriction when gluing back this M\"obius strip to the surface $\S'$. Hence we have defined a unique hyperbolic structure on the whole surface $\S$.
		\end{proof}

\subsection{Intersections}
	The following theorem generalises the well-known Ptolemy relations for arcs to the case of arbitrary curves in $\SM$. An interesting consequence is that, by ``resolving intersections'' recursively, it allows one to write the $\lambda$-length of a multigeodesic with intersections as a linear combination of $\lambda$-lengths of multigeodesics consisting of pairwise compatible simple curves. This will be crucial in the proof of Theorem \ref{theorem:basis}. Note that in the orientable case, a similar result has recently appeared in \cite{MW:resolutions}.

	\begin{theorem}\label{theorem:resolution}
		Let $\alpha$ be a multigeodesic with an intersection point $p$. Then we can write
		$$\lambda(\alpha) =  \varepsilon_1  \lambda (\beta) + \varepsilon_2  \lambda (\gamma), $$
		where $\beta$ and $\gamma$ are the two multigeodesics obtained by resolving the intersection at $p$, and $\varepsilon_1 , \varepsilon_2 \in \{ -1 , 1 \}$ are functions depending only on the topological type of $\alpha$, $\beta$ and $\gamma$ (see Figure \ref{fig:resolution}).
	\end{theorem}

	\begin{figure}
		\begin{center}
			\begin{tikzpicture}[scale = .4]

				\draw[dashed] (-19,2) -- (-17,2);
				\draw[dashed] (-19,-2) -- (-17,-2);

				\fill (-15,-2) node {$\alpha$};

				\draw (-17,2) .. controls (-16,2) and (-14,-2) .. (-13,-2);
				\draw (-17,-2) .. controls (-16,-2) and (-14,2) .. (-13,2);

				\fill (-14,0) node {$p$};

				\draw[dashed] (-13,2) -- (-11,2);
				\draw[dashed] (-13,-2) -- (-11,-2);

				\fill (-10,0) node {$= \pm$};

				\draw[dashed] (-9,2) -- (-7,2);
				\draw[dashed] (-9,-2) -- (-7,-2);

				\draw (-7,2) .. controls (-7,2) and (-6,2) .. (-5.5,1);
				\draw (-7,-2) .. controls (-7,-2) and (-6,-2) .. (-5.5,-1);

				\draw (-4.5,1) .. controls (-4,2) and (-3,2) .. (-3,2);
				\draw (-4.5,-1) .. controls (-4,-2) and (-3,-2) .. (-3,-2);

				\draw (-4.5,1) .. controls (-4.75,0) and (-5.25,0) .. (-5.5,1);
				\draw (-4.5,-1) .. controls (-4.75,0) and (-5.25,0) .. (-5.5,-1);

				\fill (-5,-2) node {$\beta$};

				\draw[dashed] (-3,2) -- (-1,2);
				\draw[dashed] (-3,-2) -- (-1,-2);

				\fill (0,0) node {$\pm$};

				\draw[dashed] (9,2) -- (7,2);
				\draw[dashed] (9,-2) -- (7,-2);

				\fill (5,-2) node {$\gamma$};

				\draw (7,2) .. controls (4.75,1) and (4.75,-1) .. (7,-2);
				\draw (3,2) .. controls (5.25,1) and (5.25,-1) .. (3,-2);

				\draw[dashed] (3,2) -- (1,2);
				\draw[dashed] (3,-2) -- (1,-2);
			\end{tikzpicture}
		\end{center}
		\caption{Resolving an intersection of a multigeodesic.}\label{fig:resolution}
	\end{figure}
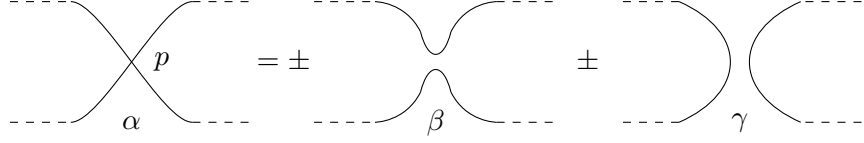

	\begin{rmq}
		In this identity, we consider $\lambda(\alpha)$, $\lambda (\beta)$ and $\lambda(\gamma)$ as functions on the Teichm\"uller space.
	\end{rmq}

	This theorem is a generalisation of both the Ptolemy relation between simple arcs and the trace identities for matrices in $\SL (2 , \mathbb{C})$. The generalisations are probably well-known to the specialists, however there are several cases for which there seems to be no reference in the literature. Moreover, in order to keep things self-contained, we give a complete proof even for classical situations.

	\begin{proof}
		First notice that the resolution of an intersection at a point $p$ only modifies the elements of the multigeodesic crossing at $p$. Hence, we only have to show the relation for multigeodesics with only one or two elements, and the general result will hold by induction. For the rest of the proof, let $\wt{\sigma} \in \wt{\mathcal{T}}(\S)$ be a decorated hyperbolic structure. We will omit the subscript and write $\lambda_{\wt{\sigma}} (a)= \lambda (a)$.

		\subsubsection{Two distinct arcs}
			Let $\alpha = \{a , b\}$ with $a$ and $b$ be two different arcs with endpoints $a_0 , a_1 , b_0 , b_1$ (not necessarily all disjoint), intersecting at some point $p \in \S$. 

			Choose a lift of $\wt{p} \in \wt{S} = \Hy$ and denote by $\wt{a}$ and $\wt{b}$ the two unique ideal decorated geodesics that pass through $\wt{p}$. This defines four different endpoints that we denote $\wt{a_0} , \wt{a_1} ,\wt{b_0} , \wt{b_1}$. These four points are necessarily disjoint (even if they are lifts of the same point in the surface) so this gives rise to a quadrilateral with sides 
			$$\wt{c} = (\wt{a_0} , \wt{b_0}), \hspace{0.5cm} \wt{d} = (\wt{b_0} , \wt{a_1}), \hspace{0.5cm} \wt{e}  = (\wt{a_1} , \wt{b_1}), \hspace{0.5cm} \wt{f}  = (\wt{b_1} , \wt{a_0}) .$$

			The diagonals of this quadrilateral are $\wt{a}$ and $\wt{b}$. The Ptolemy relation in $\Hy$ gives  
			$$\lambda (\wt{a}) \lambda (\wt{b}) = \lambda (\wt{c})\lambda (\wt{e}) + \lambda(\wt{d})\lambda(\wt{f}).$$
			When returning to the surface, the arc $c$ is the projection of $\wt{c}$. It is homotopic to the arc starting at $a_0$ following $a$ until it reaches $p$ and then following $b$ until it reaches $b_0$. The same applies to the arcs $d$, $e$ and $f$. By definition of the $\lambda$-length of the arcs $a,b,c,d,e,f$ we have
			\begin{equation*}
			\lambda(a) \lambda(b) = \lambda (c) \lambda(e) + \lambda (d) \lambda (f).
			\end{equation*}

			The resolution of the intersection gives $\beta = \{c, e\} $ and $\gamma = \{d ,f\}$.

		\subsubsection{Two closed curves}
			Let $\alpha = \{a ,b\} $ with $a$ and $b$ two distinct geodesic curves intersecting at some point $p \in \S$. We denote $\mathbf{a}$ and $\mathbf{b}$ the corresponding elements of the fundamental group $\pi_1 (\S , p)$ of the surface based at $p$ up to a choice of orientation of the curves. The holonomy map of the hyperbolic structure sends $\mathbf{a}$ and $\mathbf{b}$ to elements $A$ and $B$ of $\pgl$. We take matrix representatives in $\mbox{GL}(2, \R)$ such that $|\det (A)| = |\det (B)| = 1$ and $\tr (A) , \tr (B) > 0$. 

			The following formula holds for all such matrices :
			$$ \tr(A) \tr (B) = \tr(AB) + \det (B) \tr (AB^{-1}).$$

			The $\lambda$-length of $\alpha$ is given by $\lambda(\alpha) = |\tr (A)| |\tr(B)|$. The matrices $AB$ and $AB^{-1}$ correspond to the holonomy of the curves $a \ast b$ and $a \ast b^{-1}$. These curves are exactly the ones given by the resolution of the intersection at $p$. So we have $\lambda (\beta) = |\tr (AB)|$ and $\lambda (\gamma) = |\tr (A B^{-1})|$.

			So it is clear that there exists $\varepsilon_1$ and $\varepsilon_2$ in $\{ -1 , 1 \}$ such that
			$$\lambda(\alpha) = \varepsilon_1 \lambda(\beta) + \varepsilon_2 \lambda(\gamma).$$
			The only thing to show is that the elements $\varepsilon_1$ and $\varepsilon_2$ do not depend on the choice of the decorated hyperbolic structure $\wt{\sigma}$. To do that we use a continuity argument. 

			The functions $\tr(AB)$ and $\tr(AB^{-1})$ are continuous on the decorated Teichm\"uller space. For any given hyperbolic structure $\sigma$ and any element $c \in \pi_1 (\S)$, we have $\tr(\rho(c)) \neq 0$ where $\rho$ is the holonomy representation. Indeed, if $c$ is a two-sided curve, then the $\rho(c)$ is a hyperbolic or parabolic isometry, and hence we have $|\tr(\rho(c))| \geq 2$. If $c$ is a one-sided curve, then $\rho(c)$ is a glide-reflection. A glide-reflection with zero trace corresponds to a plain reflection which is an involution. This would contradict the faithfulness of the holonomy representation.

			As $\wt{\mathcal{T}}\SM$ is connected, the signs of $\tr(AB)$ and $\tr(AB^{-1})$ are constant. And hence $\varepsilon_1$ and $\varepsilon_2$ only depend on the geometric type of $\alpha$, $\beta$ and $\gamma$.

		\subsubsection{One non-simple closed curve}
			Let $\alpha = \{ c \}$ with $c$ a non-simple closed curve having an auto-intersection at the point $p$ in the interior of $\S$. We can see $c$ as an element of the fundamental group based at $p$. The curve $c$ can henceforth be written as $a \ast b$ with $a$ and $b$ be the two parts of the curve when removing the point $p$. This corresponds to one of the resolution, so set $\beta = \{a , b\}$. The other resolution of the intersection is the curve $\gamma = a \ast b^{-1}$ which has at least one self-intersection less than $c$. A simple permutation of the terms of the preceding case gives :
			$$\lambda(\alpha) = \varepsilon_1 \lambda(\beta) + \varepsilon_2 \lambda(\gamma).$$

		\subsubsection{One arc and one curve}\label{subsub:onearconecurve}
			Let $\alpha = \{a , b\} $ with $a$ an arc and $b$ a closed curve intersecting each other at $p \in \S$. The curve $b$ corresponds to an element $\mathbf{b} \in \pi_1 (\S)$. We lift everything in the universal cover $\Hy$. 

			$\bullet$ First assume that $b$ is a two-sided curve. Up to conjugacy and rescaling the isometry $\rho (\mathbf{b})$ is given by the following matrix :
			$$B = \rho (\mathbf{b}) = \left( \begin{array}{cc} \eta & 0 \\ 0 & \frac{1}{\eta} \end{array} \right)$$
			with $\eta > 1$ so that $$\lambda (b) = |\tr (\rho (\mathbf{b}) )| = \eta + \frac{1}{\eta}.$$ The axis of such an isometry is the vertical axis $x=0$ and the direction is given by the positive direction in $y$.

			Let $\wt{p}$ be a lift of $p$ on the axis $x=0$, and let $\wt{a}$ be the unique decorated geodesic that is a lift of $a$ passing through $\wt{p}$. We denote by $U = (u,h)$ and $V = (v,k)$ the horocycles defining this decorated geodesic. The geodesic crosses the vertical axis $x=0$ and hence $u$ and $v$ will be disjoint from $0$ and $\infty$ and without loss of generality we can choose $u < 0$ and $v > 0$, so that we have~:
			$$ \lambda (a) = \dfrac{v - u}{\sqrt{hk}}.$$

			The image of the horocycles $U$ and $V$ under the isometry $B = \rho (\mathbf{b})$ are given by
			$$B (U) = (\eta^2 u , \eta^2 h), \quad B (V) = (\eta^2 v , \eta^2 k).$$
			It is easy to check that the $\lambda$-length of the decorated geodesic $(B (U) , B (V))$ is still $\lambda (a)$. Let $\wt{e}$ and $\wt{f}$ be the decorated geodesics corresponding to $(U , B(V))$ and $(B(U) , V)$ respectively. These geodesics correspond to arcs $e$ and $f$ on $\S$. We have
			$$\lambda (\wt{e}) = \dfrac{\eta^2 v - u}{\eta \sqrt{hk}}, \quad \lambda (\wt{f}) = \dfrac{v-\eta^2 u}{\eta \sqrt{hk}}. $$
			These arcs correspond to the resolution of the intersection at $p$, hence we can note $\beta = e$ and $\gamma = f$. We then have the following relation :
			\begin{align*}
			\lambda(\alpha) = \lambda (a) \lambda (b) & = \dfrac{v - u}{\sqrt{hk}} \left( \eta + \dfrac{1}{\eta} \right) \\
			& = \dfrac{\eta v }{\sqrt{hk}} + \dfrac{ v }{\eta \sqrt{hk}} - \dfrac{\eta u }{\sqrt{hk}} - \dfrac{u }{\eta \sqrt{hk}}\\
			& = \dfrac{\eta^2 v - u}{\eta \sqrt{hk}} + \dfrac{v-\eta^2 u}{\eta \sqrt{hk}} \\
			& = \lambda (\beta) + \lambda (\gamma). 
			\end{align*}

			$\bullet$ Now, if $b$ is a one-sided curve. Up to conjugacy the isometry $\rho (\mathbf{b})$ is given by the following matrix 
			$$\rho (\mathbf{b}) = \left( \begin{array}{cc} \eta & 0 \\ 0 & - \frac{1}{\eta} \end{array} \right)$$
			with $\eta > 1$ so that $$\lambda (b) = \eta - \frac{1}{\eta}.$$
			Again, let $\wt{p}$ be a lift of $p$ on the axis $x=0$, and let $\wt{a}$ be the unique decorated geodesic that is a lift of $a$ passing through $\wt{p}$. We denote by $U = (u,h)$ and $V = (v,k)$ the horocycles defining this decorated geodesic with $u < 0$ and $v > 0$, so that we have~:
			$$ \lambda (a) = \dfrac{v - u}{\sqrt{hk}}.$$
			The image of the horocycles $U$ and $V$ under the isometry $B = \rho (\mathbf{b})$ are given by
			$$B (U) = (- \eta^2 u , \eta^2 h), \quad  B (V) = (- \eta^2 v , \eta^2 k).$$
			
			Let $\wt{e}$ and $\wt{f}$ be the decorated geodesics corresponding to $(U , B(V))$ and $(B(U) , V)$ respectively. These geodesics correspond to arcs $e$ and $f$ on $\S$. We have
			$$\lambda (\wt{e}) = \dfrac{ u + \eta^2 v}{\eta \sqrt{hk}}, \quad \lambda (\wt{f}) = - \dfrac{v+\eta^2 u}{\eta \sqrt{hk}} .$$
			So finally we have the relation :
			\begin{align*}
			\lambda (\alpha) = \lambda (a ) \lambda (b) & = \dfrac{v - u}{\sqrt{hk}} \left( \eta - \dfrac{1}{\eta} \right) \\
			& = \dfrac{\eta^2 v + u}{\eta \sqrt{hk}} - \dfrac{v+\eta^2 u}{\eta \sqrt{hk}} \\
			& = \lambda (\beta) + \lambda (\gamma) .
			\end{align*}

		\subsubsection{One non-simple arc}
			Let $\alpha = a$ with $a$ a non-simple arc with a self-intersection at a point $p \in \S$. Then we can define a closed curve $b$ based at the point $p$ which correspond to the loop created by $a$. Let $\mathbf{b}$ be the corresponding element of $\pi_1 (\S , p)$.

			$\bullet$ If $b$ is two-sided, then up to conjugacy we have
			$$B = \rho (\mathbf{b}) = \left( \begin{array}{cc} \eta & 0 \\ 0 & \frac{1}{\eta} \end{array} \right).$$

			Let $\wt{a}$ be the decorated geodesic corresponding to a lift of the arc and let $U = (u,h)$ and $V = (v,k)$ be two horocycles such that $\wt{a} = (U , B(V) )$ and choose $v>u$. In this setting we have necessarily that the decorated geodesic $\wt{a}_{-}$ corresponding to $(B^{-1}(U) , V)$ intersect the geodesic $\wt{a}$ at a point $\wt{p}_{-}$ and similarly the geodesic $\wt{a}_{+}$ intersect $\wt{a}$ at $\wt{p}_{+}$.

			This implies that the geodesic $\wt{a}$ does not cross the vertical axis $x=0$ and hence $u$ and $v$ are of the same sign. Without loss of generality, we may assume that $u, v >0$.

			Define $\wt{c}_-$ to be the decorated geodesic $(U , V)$ and $\wt{c}_+$ to be the decorated geodesic $(B(U) , B(V))$. Clearly, these two geodesics are lifts of the same arc $c$ in $\S$. Define also $\wt{d}$ to be the decorated geodesic $(V , B(U))$. So we have~:
			$$\lambda(a) = \dfrac{\eta^2 v - u}{\eta \sqrt{hk}} , \hspace{3pt} \lambda(c) = \dfrac{ v - u}{\sqrt{hk}}, \hspace{3pt} \lambda(c) = \dfrac{ v - \eta^2 u }{\eta \sqrt{hk}}, \hspace{3pt} \lambda (b) = \eta + \dfrac{1}{\eta}$$

			The resolutions at point $p$ are given by the multigeodesic $\beta = b \sqcup c$ and $\gamma = d$. Calculations similar to the preceding case show that
			$$\lambda(\alpha) = \lambda(b) \lambda(c) + \lambda (d) = \lambda(\beta) + \lambda(\gamma)$$.

			\medskip

			$\bullet$ If $b$ is one-sided, then up to conjugacy we have
			$$B = \rho (\mathbf{b}) = \left( \begin{array}{cc} \eta & 0 \\ 0 & - \frac{1}{\eta} \end{array} \right)$$

			Let $\wt{a}$ be the decorated geodesic corresponding to a lift of the arc and let $U = (u,h)$ and $V = (v,k)$ be two horocycles such that $\wt{a} = (U , B(V) )$ and choose $v>u$. In this setting we have necessarily that the decorated geodesic $\wt{a}_{-}$ corresponding to $(B^{-1}(U) , V)$ intersect the geodesic $\wt{a}$ at a point $\wt{p}_{-}$ and similarly the geodesic $\wt{a}_{+}$ intersect $\wt{a}$ at $\wt{p}_{+}$.

			This implies that the geodesic $\wt{a}$ cross the vertical axis $x=0$ and hence $u$ and $B(v)$ are of different sign. As $B$ is orientation reversing, we have that $v$ and $B(v)$ are of different sign, and hence without loss of generality, we may assume that $v > u >0 $.

			Define $\wt{c}_-$ to be the decorated geodesic $(U , V)$ and $\wt{c}_+$ to be the decorated geodesic $(B(U) , B(V))$. Clearly, these two geodesics are lifts of the same arc $c$ in $\S$. Define also $\wt{d}$ to be the decorated geodesic $(V , B(U))$. So we have 
			$$\lambda(a) = \dfrac{u + \eta^2 v}{\eta \sqrt{hk}} , \hspace{3pt} \lambda(c) = \dfrac{ v - u}{\sqrt{hk}}, \hspace{3pt} \lambda(d) = \dfrac{ v + \eta^2 u }{\eta \sqrt{hk}}, \hspace{3pt} \lambda (b) = \eta - \dfrac{1}{\eta}$$

			Again, the resolutions at point $p$ are given by the multigeodesic $\beta = \left\{ b , c\right\}$ and $\gamma = d$. Calculations similar to the preceding case show that
			$$\lambda(\alpha) = \lambda(b) \lambda(c) + \lambda (d) = \lambda (\beta) + \lambda(\gamma).$$
	\end{proof}

	\begin{rmq}
		The coefficients $\varepsilon_1$ and $\varepsilon_2$ are always $+1$ except in the case where the crossing involves only  closed curves and no arcs. In this case, the coefficients cannot be both negative at the same time because the left term of the identity is necessarily positive. The computation of the coefficient for a given situation can be done by taking one example of a hyperbolic structure and computing the traces. By continuity and connexity argument, the value for one example will be the value for all Teichm\"uller space.

		For example if $\alpha = \{a , b\}$ with $a$ and $b$ two simple closed two-sided curves that intersect only once, then $\varepsilon_1 = \varepsilon_2 = +1$. This is proved using the fact that the commutator $a \ast b \ast a^{-1} \ast b^{-1}$ bounds a one-holed torus embedded in $\S$. Explicit examples of hyperbolic structure on a one-holed torus are classical and, using one particular hyperbolic structure, we see that the signs are all positive. 
	\end{rmq}

	The case of a multigeodesic consisting of a unique one-sided curve with multiplicity more than one, has to be treated separately. Indeed, any two curves homotopic to a one-sided curve will have at least one intersection point. Recall that in the orientable case, two homotopic two-sided curves can always be made disjoint.
	\begin{prop}\label{prop:d2} 
		Let $d$ be a one-sided closed curve corresponding to an element $\mathbf{d} \in \pi_1 (\S)$. Then
		$$\lambda (d)^2 = \lambda(e) - 2$$
		where $e$ is the two-sided closed curve corresponding to the element $\mathbf{d}^2 \in \pi_1 (\S)$.
	\end{prop}
	\begin{proof}
		Let $\wt{\sigma} \in \wt{\mathcal{T}}\SM$ and let $D = \rho_\sigma (\mathbf{d})$. Up to conjugacy and rescaling, the matrix $D$ is given by
		$$D = \left( \begin{array}{cc} \mu & 0 \\ 0 & -1/\mu\end{array} \right) , \mbox{ with } \mu > 1$$
		The $\lambda$-length of the multigeodesic $\alpha = \{d,d \}$ is given by
		$$\lambda(\alpha) = \lambda (d)^2 = \tr(D)^2$$
		On the other hand, the $\lambda$-length of $e$ given by $\lambda(e) = |\tr (\rho(\mathbf{d}^2))| = |\tr (\mathbf{d}^2) |$ and hence
		$$\lambda (e) = |\tr (D^2)| = |\mu^2 + \dfrac{1}{\mu^2}| = \left(\mu - \dfrac{1}{\mu}\right)2 + 2 = \lambda(d)^2 +2 = \lambda(\alpha) + 2.$$
	\end{proof}

	\begin{rmq}\label{rmq:d2}
		Slightly abusing notations, we can restate Proposition \ref{prop:d2} by saying that 
		$$\lambda(d^2) = \lambda(d)^2 +2$$
		for any one-sided closed curve $d$ in $\SM$, see Figure \ref{fig:detd2}. This identification will be of particular use in the proof of Theorem \ref{theorem:basis}.
	\end{rmq}

	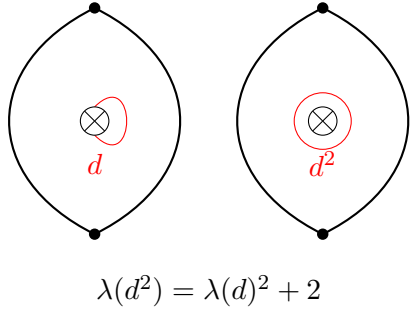
\begin{figure}
		\begin{center}
			\begin{tikzpicture}[scale = .75]
				\arcd 0 0
				\mobdeux 0 0
				\fill[red] (0,-.75) node {$d$};

				\draw[red] (4,0) circle (.5);
				\mobdeux 4 0
				\fill[red] (4,-.75) node {$d^2$};

				\fill (2,-3) node {$\lambda(d^2) = \lambda(d)^2 + 2$};
		
			\end{tikzpicture}
		\end{center}
		\caption{Relations between $\lambda(d)$ and $\lambda(d^2)$ for the M\"obius strip with two marked points.}\label{fig:detd2}
	\end{figure}

\section{Quasi-cluster algebras associated with non-orientable surfaces}
	In this section $\SM$ is an unpunctured marked surface of rank $n \geq 1$ with $b \geq 1$ boundary segments and $\mathcal F$ is a field of rational functions in $n+b$ indeterminates. To any boundary segment $b$ in $\B\SM$ we associate a variable $x_b \in \mathcal F$ such that $\ens{x_b \ | \ b \in \B\SM}$ is algebraically independent in $\mathcal F$ and we set 
	$$\Z\P = \Z[x_b^{\pm 1} \ | b \in \B\SM] \subset \mathcal F,$$
	which is referred to as the \emph{ground ring}.

	\subsection{Quasi-seeds and their mutations}
		\begin{defi}
			A \emph{quasi-seed} associated with $\SM$ in $\mathcal F$ is a pair $\Sigma = (T,\x)$ such that~:
			\begin{enumerate}
				\item $T$ is a quasi-triangulation of $\SM$~;
				\item $\x = \ens{x_{t} \ | \ t \in T}$ is a free generating set of the field $\mathcal F$ over $\Z\P$. 
			\end{enumerate}
			The set $\ens{x_{t} \ | \ t \in T}$ is called the \emph{quasi-cluster} of the quasi-seed $\Sigma$.

			A quasi-seed is called a \emph{seed} if the corresponding quasi-triangulation is a triangulation and in this case the quasi-cluster is called a \emph{cluster}.
		\end{defi}

		\begin{defi}\label{defi:qmut}
			Given $t \in T$, we define the \emph{quasi-mutation of $\Sigma$ in the direction $T$} as the pair $\mu_t(T,\x) = (T',\x')$ where $T' = \mu_t(T) = T \setminus \ens{t} \sqcup \ens{t'}$ and $\x'=\ens{x_v \ | \ v \in T'}$ such that $x_{t'}$ is defined as follows~:
			\begin{enumerate}
				\item If $t$ is an arc separating two different triangles with sides $(a,b,t)$ and $(c,d,t)$ as in the figure below,
					
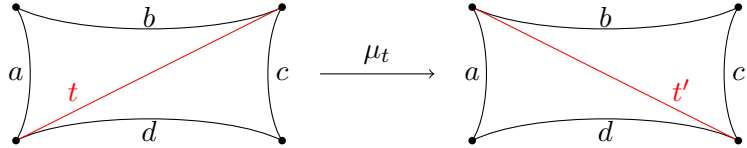
\begin{center}
						\begin{tikzpicture}[scale = .5]
							\draw (.5,.23) .. controls (2,1) and (6,1) .. (7.5,.23);
							\draw (.5,3.77) .. controls (2,3) and (6,3) .. (7.5,3.77);

							\draw (.5,.23) .. controls (1,1) and (1,3) .. (.5,3.77);	
							\draw (7.5,.23) .. controls (7,1) and (7,3) .. (7.5,3.77);	

							\draw[red] (.5,.23) -- (7.5,3.77);

							\fill (.5,.23) circle (.1);
							\fill (.5,3.77) circle (.1);
							\fill (7.5,3.77) circle (.1);
							\fill (7.5,.23) circle (.1);

							\fill (.5,2) node {$a$};
							\fill (7.5,2) node {$c$};
							\fill (4,3.5) node {$b$};
							\fill (4,.5) node {$d$};

							\fill[red] (2,1) node [above] {$t$};
						

							\draw[->] (8.5,2) -- (11.5,2);
							\fill (10,2) node [above] {$\mu_t$};

						
							\draw (12.5,.23) .. controls (14,1) and (18,1) .. (19.5,.23);
							\draw (12.5,3.77) .. controls (14,3) and (18,3) .. (19.5,3.77);

							\draw (12.5,.23) .. controls (13,1) and (13,3) .. (12.5,3.77);	
							\draw (19.5,.23) .. controls (19,1) and (19,3) .. (19.5,3.77);	

							\draw[red] (12.5,3.77) -- (19.5,.23);

							\fill (12.5,.23) circle (.1);
							\fill (12.5,3.77) circle (.1);
							\fill (19.5,3.77) circle (.1);
							\fill (19.5,.23) circle (.1);

							\fill (12.5,2) node {$a$};
							\fill (19.5,2) node {$c$};
							\fill (16,3.5) node {$b$};
							\fill (16,.5) node {$d$};

							\fill[red] (18,1) node [above] {$t'$};
						\end{tikzpicture}
					\end{center}
					then the relation is simply given by the Ptolemy relation for arcs, that is,
					$$x_t x_{t'}= x_a x_c + x_b x_d.$$
				\item If $t$ is an arc in an anti-self-folded triangle with sides $(t,t,a)$, as in the figure below,
					
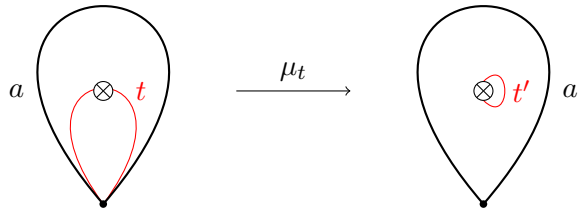
\begin{center}
						\begin{tikzpicture}[scale = .5]

							\arcaone 0 0 
							\mobone 0 0
							\fill[red] (1,0) node {$t$};
							
							\fill (-1.8,0) node [left] {$a$};
						

							\draw[->] (3.5,0) -- (6.5,0);
							\fill (5,0) node [above] {$\mu_t$};

							
							\arcd {10} 0
							\mobone {10} 0
							\fill[red] (11,0) node [] {$t'$};
						
							\fill (11.8,0) node [right] {$a$};

						\end{tikzpicture}
					\end{center}
					then the relation is
					$$x_t x_{t'} = x_a.$$
				\item If $t$ is a one-sided curve in an annuli with boundary $a$ as in the figure below, 
					
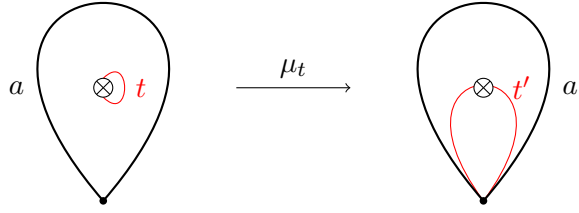
\begin{center}
						\begin{tikzpicture}[scale = .5]

							\arcd {0} 0
							\mobone 0 0
							\fill[red] (1,0) node [] {$t$};
						
							\fill (-1.8,0) node [left] {$a$};


							\draw[->] (3.5,0) -- (6.5,0);
							\fill (5,0) node [above] {$\mu_t$};

							
							\arcaone {10} 0
							\mobone {10} 0
							\fill[red] (11,0) node {$t'$};
						
							\fill (11.8,0) node [right] {$a$};

						\end{tikzpicture}
					\end{center}
					then the relation is 
					$$x_t x_{t'} = x_a.$$
				\item If $t$ is an arc separating a triangle with sides $(a,b,t)$ and an annuli with boundary $t$ and one-sided curve $d$ as in the figure below,
					\def\arcdnoir#1#2{
						\draw (#1,#2+.25) .. controls (#1+.75,#2+1) and (#1+.75,#2-1) .. (#1,#2-.25);
					}

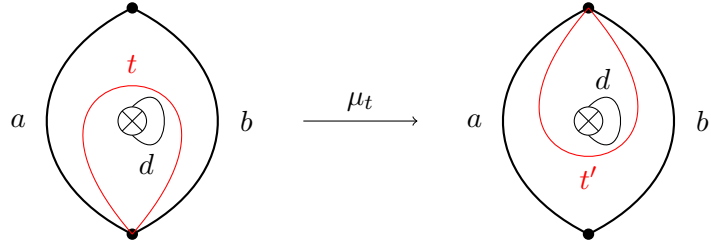
					\begin{center}
						\begin{tikzpicture}[scale = .75]

							\arcdnoir {0} 0
							\mobdeux 0 0
							\arcca 0 0
							\fill[red] (0,1) node [] {$t$};

							\fill (-2,0) node {$a$};
							\fill (2,0) node {$b$};

							\fill (.25,-.75) node {$d$};
						

							\draw[->] (3,0) -- (5,0);
							\fill (4,0) node [above] {$\mu_t$};

							
							\arcdnoir {8} 0
							\mobdeux {8} 0
							\arccb {8} 0
							\fill[red] (8,-1) node [] {$t'$};

							\fill (8-2,0) node {$a$};
							\fill (8+2,0) node {$b$};

							\fill (8.25,.75) node {$d$};

						\end{tikzpicture}
					\end{center}
					then the relation is
					$$x_t x_{t'} = (x_a + x_b)^2 + x_d^2 x_a x_b.$$
			\end{enumerate}
		\end{defi}
		Note that the quasi-mutation of a quasi-seed is again a quasi-seed.

		Two quasi-seeds $\Sigma=(T,\x)$ and $\Sigma'=(T',\x')$ associated with $\SM$ in $\mathcal F$ are called \emph{quasi-mutation-equivalent} if $\Sigma'$ can be obtained from $\Sigma$ by a finite number of quasi-mutations. This defines an equivalence relation on the set of seeds associated with $\SM$ in $\mathcal F$ whose equivalence classes are called \emph{quasi-mutation classes}.

		Since $\SM$ has rank $n$, every quasi-triangulation $T$ in $\SM$ has $n$ elements. We can thus fix a labelling $t_1, \ldots, t_n$ of the elements of $T$. A quasi-seed $\Sigma$ equipped with such a labelling is called a \emph{labelled quasi-seed}. For any $1 \leq i \leq n$, we define the \emph{mutation in the direction $i$} of the labelled quasi-seed $\Sigma$ as $\mu_i(\Sigma) = \mu_{t_i}(\Sigma)=(T',\x')$ equipped with the labelling $T' = \ens{t_1', \ldots, t_n'}$ where $t_k'=t_k$ if $k \neq i$ and $t_i'$ is the quasi-flip of $t_i$ with respect to $T$. Note that mutations of labelled quasi-seeds are involutive in the sense that $\mu_i(\mu_i(\Sigma)) = \Sigma$ for any $1 \leq i \leq n$.

	\subsection{Quasi-cluster algebras}
		Let $\mathbb T_n$ denote the $n$-regular tree. At each vertex in $\mathbb T_n$, we label by $\ens{1, \ldots, n}$ the $n$ adjacent edges. 

		\begin{defi}
			A \emph{quasi-cluster pattern} associated with $\SM$ in $\mathcal F$ is an assignment $\mathcal X: v \mapsto \Sigma_v$ for each vertex $v$ of $\mathbb T_n$ where $\Sigma_v=(T_v,\x_v)$ is a labelled quasi-seed associated with $\SM$ in $\mathcal F$ and where two adjacent quasi-seeds in $\mathbb T_n$ are related by a single mutation in the sense that 
			$$\xymatrix{\Sigma_v \ar@{-}[r]^k & \Sigma_{v'}} \text{ in } \mathbb T_n \Iff \Sigma_{v'} = \mu_k(\Sigma_v).$$
		\end{defi}

		\begin{defi}
			Let $\mathcal X : v \mapsto \Sigma_v$ be a quasi-cluster pattern associated with $\SM$ in $\mathcal F$. The \emph{quasi-cluster algebra associated with $\mathcal X$} is the $\Z\P$-subalgebra $\mathcal A(\mathcal X)$ of $\mathcal F$ generated by the union of all the quasi-clusters of quasi-seeds appearing in the quasi-cluster pattern, that is, 
			$$\mathcal A(\mathcal X) = \Z\P\left[x \ | \ x \in \bigcup_{v} \x_v\right]$$
			where $v$ runs over the vertices in $\mathbb T_n$.

			The elements in the union of all the quasi-clusters of quasi-seeds appearing in the quasi-cluster pattern are called the \emph{quasi-cluster variables} of the quasi-cluster algebra $\mathcal A(\mathcal X)$.
		\end{defi}

		Note that each labelled quasi-seed $\Sigma$ associated with $\SM$ in $\mathcal F$ determines entirely a quasi-cluster pattern $\mathcal X$ (up to a relabelling of the vertices in $\mathbb T_n$) so that the quasi-cluster algebra $\mathcal A(\mathcal X)$ is entirely determined by $\Sigma$ and is denoted by $\mathcal A_{\Sigma}$. Note also that different choices of labelling of a quasi-seed $\Sigma$ associated with $\SM$ in $\mathcal F$ give rise to canonically isomorphic quasi-cluster algebras so that we can associate a quasi-cluster algebra $\mathcal A_\Sigma$ to any quasi-seed $\Sigma$ associated with $\SM$ in $\mathcal F$.

		Finally, note that if $\Sigma=(T,\x)$ and $\Sigma'=(T',\x')$ are two quasi-seeds associated with $\SM$ in $\mathcal F$, then the quasi-triangulations $T'$ and $T$ are quasi-mutation-equivalent so that there exists a seed $\Sigma'' = (T',\x'')$ in the quasi-cluster pattern defined by $\Sigma$ and the canonical automorphism of $\mathcal F$ sending $\x''$ to $\x'$ induces an isomorphism of the quasi-cluster algebras $\mathcal A_\Sigma$ and $\mathcal A_{\Sigma'}$. Thus, up to a canonical ring isomorphism, the quasi-cluster algebra $\mathcal A_\Sigma$ only depends on the surface $\SM$ and is denoted by $\mathcal A_{\SM}$.

		\begin{defi}
			Let $\SM$ be a non-oriented unpunctured marked surface, then $\mathcal A_{\SM}$ is called the \emph{quasi-cluster algebra associated with the surface $\SM$}.
		\end{defi}

		Note that the quasi-cluster of any quasi-seed $\Sigma=(T,\x)$ in $\mathcal A_{\SM}$ is a free generating set of $\mathcal F$ over $\Z\P$ so that each quasi-cluster variable $x$ in $\mathcal A_{\SM}$ can be expressed as a rational function with coefficients in $\Z\P$ in the quasi-cluster $\x$. This rational expression is called the \emph{$\Sigma$-expansion} of $x$ in $\mathcal A_{\SM}$.

		It follows from the definition of the quasi-cluster algebra $\mathcal A_{\SM}$ that each quasi-cluster variable $x$ in $\mathcal A_{\SM}$ is associated with a quasi-arc in $\mathcal A_{\SM}$. If $\Sigma=(T,\x)$ is a quasi-seed in $\mathcal A_{\SM}$, we saw in Theorem \ref{theorem:lambdaindalg} that the $\lambda$-lengths of arcs in $T$ can be viewed as algebraically independent variables. Therefore, there is an isomorphism of $\Z$-algebras~:
		$$\phi_T: \left\{\begin{array}{rcl}
			\Q(\lambda(t) \ | \ t \in T \sqcup \B\SM) & \xrightarrow{\sim} & \mathcal F\\
			\lambda(t) & \longmapsto & x_t \text{ for any } t \in T \sqcup \B\SM
		\end{array}\right.$$ 

		\begin{lem}\label{lem:lambdax}
			Let $\SM$ be an unpunctured marked surface, let $\Sigma=(T,\x)$ be a quasi-seed in $\mathcal A_{\SM}$ and let $x$ be a quasi-cluster variable in $\mathcal A_{\SM}$ corresponding to a quasi-arc $v$ in $\qA\SM$. Then the $\Sigma$-expansion of $x$ is given by $\phi_T(\lambda(v))$.
		\end{lem}
		\begin{proof}
			Let $\Sigma'=(T',\x')$ be a quasi-seed in $\mathcal A_{\SM}$ which is quasi-mutation-equivalent to $\Sigma$. We prove by induction on the minimal number $d(\Sigma,\Sigma')$ of quasi-mutations to reach $\Sigma'$ from $\Sigma'$ that the result holds for any quasi-cluster variable in $\Sigma'$. If $\Sigma=\Sigma'$, then the result clearly holds. 

			Otherwise, we can write $\Sigma' = \mu_v(\Sigma'')$ with $d(\Sigma,\Sigma'')<d(\Sigma,\Sigma')$. Therefore, the result holds for any quasi-cluster variable in $\Sigma''$ by induction hypothesis. Let denote by $v'$ the quasi-flip of $v$ with respect to the quasi-triangulation $T''$. The quasi-mutation rules precisely imitate the relations for the $\lambda$-lengths of the corresponding arcs. This is clear for the first three cases considered in Definition \ref{defi:qmut} and for the fourth case, it follows from the resolution of the two intersections of the corresponding arcs and from the identity given in Proposition \ref{prop:d2}. As $x_v x_{v'} = M_1 + M_2$ where $M_1$ and $M_2$ are monomials in the variables corresponding to the quasi-arcs in $T''$, applying $\phi_T$ to the corresponding relation for $\lambda(v) \lambda(v')$ and using the induction hypothesis, we get $\phi_T(\lambda(v')) = \frac{M_1 + M_2}{x_v} =x_{v'}$.
		\end{proof}
 
		Therefore, quasi-cluster variables in $\mathcal A_{\SM}$ are indexed by elements of $\qA\SM$ and we denote by $\ens{x_a \ | \ a \in \qA\SM}$ the set of quasi-cluster variables in $\mathcal A_{\SM}$ so that 
		$$\mathcal A_{\SM} = \Z\P[x_a \ | \ a \in \qA\SM] \subset \mathcal F.$$
		By definition, the \emph{cluster variables} in $\mathcal A_{\SM}$ are the quasi-cluster variables in $\mathcal A_{\SM}$ corresponding to arcs in $\SM$. In other words, the cluster variables in $\mathcal A_{\SM}$ are the elements $x_a$ with $a \in \A\SM$. Note that using Lemma \ref{lem:lambdax}, we will usually identify quasi-cluster variables with $\lambda$-lengths of the corresponding quasi-arcs.

		\begin{exmp}
			In Figure \ref{fig:variablesM2}, we exhibit the quasi-variables in the quasi-cluster algebra $\mathcal A_{\mathcal M_2}$ expressed in a particular quasi-cluster which does not correspond to a triangulation. For simplicity, for any quasi-arc $v$ in $\mathcal M_2$, we designated the quasi-cluster variable $x_v$ by $v$.
			\begin{figure}
				\begin{center}
					\begin{tikzpicture}[scale = .75]
				%
				%

						\arcca {-6} 0 
						\arcd {-6} 0 
						\mobdeux {-6} 0 

						\fill[red] (-6,-.75) node {$d$};
						\fill[red] (-6,1) node {$c_a$};

						\fill (-7.5,-1.5) node {$y$};
						\fill (-4.5,-1.5) node {$z$};

						\draw[thick,->] (-4,0) -- (-2,0);
						\fill (-3,0) node [above] {$\mu_{c_a}$};

						\arcd {0} 0 
						\arccb 0 0 
						\mobdeux 0 0 

						\fill[red] (0,.75) node {$d$};
						\fill[red] (0,-1) node {$c_b$};

						\fill (-1.5,-1.5) node {$y$};
						\fill (1.5,-1.5) node {$z$};

						\fill (0,3) node {$\displaystyle c_b = \frac{z^2+2zy+y^2+d^2zy}{c_a}$};

						\draw[thick,->] (2,0) -- (4,0);
						\fill (3,0) node [above] {$\mu_d$};

						\arcb {6} 0 
						\arccb {6} 0 
						\mobdeux {6} 0 

						\fill[red] (6,.75) node {$b$};
						\fill[red] (6,-1) node {$c_b$};

						\fill (6-1.5,-1.5) node {$y$};
						\fill (6+1.5,-1.5) node {$z$};

						\fill (6,3) node {$\displaystyle b = \frac{z^2+2zy+y^2+d^2zy}{c_ad}$};

						\draw[thick,->] (6,-3) -- (6,-5);
						\fill (6,-4) node [right] {$\mu_{c_b}$};

						\arcb {6} {-8} 
						\arcc {6} {-8} 
						\mobdeux {6} {-8}

						\fill[red] (6.75,-7.5) node {$b$};
						\fill[red] (6,-9) node [right] {$c$};

						\fill (6-1.5,-9.5) node {$y$};
						\fill (6+1.5,-9.5) node {$z$};

						\fill (6,-11) node {$\displaystyle c = \frac{z+y}{d}$};

						\draw[thick,->] (4,-8) -- (2,-8);
						\fill (3,-8) node [below] {$\mu_b$};

						\arca {0} {-8}
						\arcc {0} {-8}
						\mobdeux {0} {-8}

						\fill[red] (.75,-8.5) node {$a$};
						\fill[red] (0,-7) node [right] {$c$};

						\fill (-1.5,-9.5) node {$y$};
						\fill (1.5,-9.5) node {$z$};

						\fill (0,-11) node {$\displaystyle a = \frac{c_a}{d}$};


						\draw[thick,->] (-2,-8) -- (-4,-8);
						\fill (-3,-8) node [below] {$\mu_c$};

						\arca {-6} {-8} 
						\arcca {-6} {-8} 
						\mobdeux {-6} {-8}

						\fill[red] (-6,-8.5) node {$a$};
						\fill[red] (-6,-7) node {$c_a$};

						\fill (-7.5,-9.5) node {$y$};
						\fill (-4.5,-9.5) node {$z$};

						\fill (-6,-11) node {$\displaystyle a = \frac{c_a}{d}$};


						\draw[thick,->] (-6,-3) -- (-6,-5);
						\fill (-6,-4) node [left] {$\mu_d$};

					\end{tikzpicture}
				\end{center}
				\caption{The quasi-exchange graph of the M\"obius strip with two marked points and the corresponding quasi-cluster variables, expressed in the quasi-cluster $(c_a,d)$.}\label{fig:variablesM2}
			\end{figure}
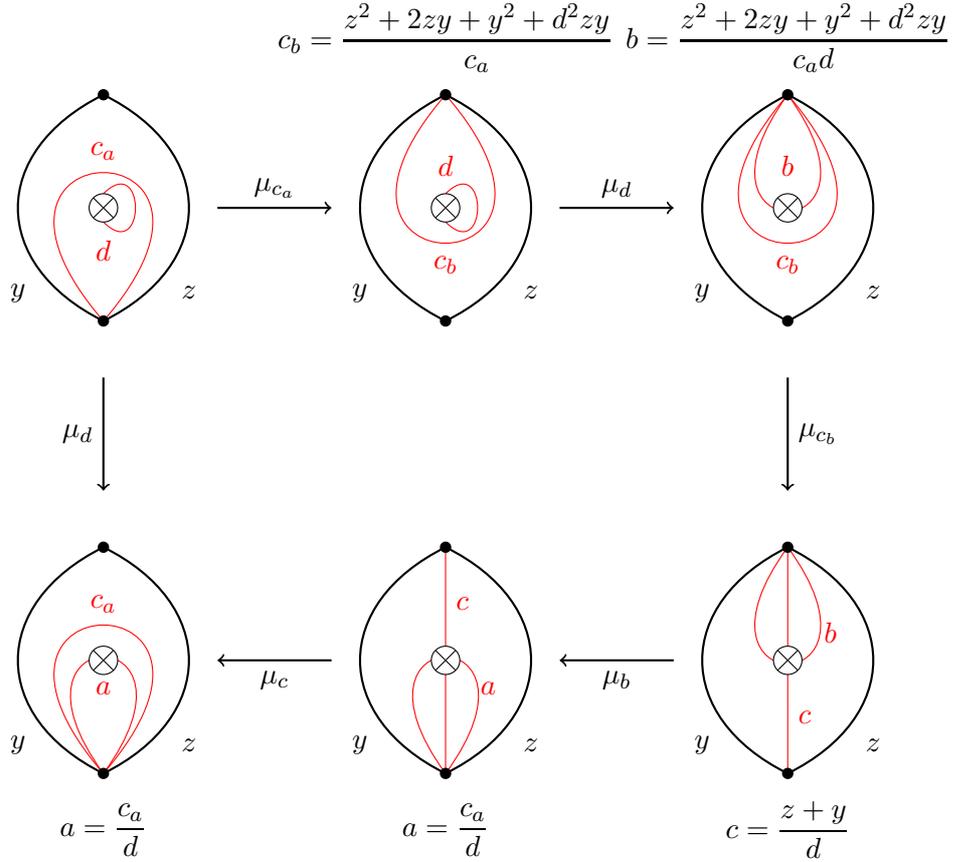
		\end{exmp}

	\subsection{Orientable vs non-orientable}\label{ssection:orientable}
		If $\SM$ is orientable, Fomin, Shapiro and Thurston associated to $\SM$ a cluster algebra in \cite{FST:surfaces,FT:surfaces2}. When the ground ring of the cluster algebra is the group ring of the free abelian group generated by variables associated to the boundary segments of $\SM$, we say that this cluster algebra has \emph{coefficients associated with the boundary segments}.

		The following proposition follows directly from the definitions and from the geometric interpretation of the cluster algebras from surfaces provided in \cite{FT:surfaces2}~:
		\begin{prop}
			Assume that $\SM$ is an orientable unpunctured marked surface and fix an orientation of $\SM$. Then the quasi-cluster algebra $\mathcal A_{\SM}$ is the cluster algebra associated with $\SM$ with coefficients associated with the boundary segments. \hfill \qed
		\end{prop}

\section{Quasi-cluster algebras and double covers}
	We saw in Section \ref{ssection:orientable} that if $\SM$ is orientable, then the quasi-cluster algebra coincides with the cluster algebra associated with $\SM$. The aim of this section is to prove that when $\SM$ is non-orientable, part of the quasi-cluster algebra structure on $\mathcal A_{\SM}$ can be found in the cluster algebra associated with the (orientable) double cover of $\SM$. Nevertheless, as we shall see, mutations in the double cover do not allow to realize quasi-cluster variables corresponding to one-sided curves. 

	Throughout this section, $\SM$ will always denote a non-orientable unpunctured marked surface of rank $n \geq 1$.

	\subsection{Lifts of triangulations and double mutations}
		We recall that each non-orientable marked surface $\SM$ admits a minimal orientable cover, its \emph{double cover} $\bSM$, endowed with a free action of $\Z_2=\ens{1,\tau}$ such that $\bSM/\Z_2 \simeq \SM$. Each element $a$ in $\A\SM$ (resp. in $\B\SM$) admits exactly two lifts $\overline a$ and $\tau \overline a$ in $\A\bSM$ (resp. in $\B\bSM$). As arcs and curves of a quasi-triangulation are simple, the element $\overline a$ and $\tau \overline a$ are compatible in $\bSM$. 


\def\crosscap#1#2#3{
	\draw[fill=white] (#1,#2) circle (#3);
	\draw (#1,#2) -- ++(45:#3);
	\draw (#1,#2) -- ++(-45:#3);
	\draw (#1,#2) -- ++(135:#3);
	\draw (#1,#2) -- ++(-135:#3);
}

\def\cc#1#2{
	\crosscap{#1}{#2}{.25}
}

\def\mobdeuxbw#1#2{
	\draw[thick] (#1,#2+2) .. controls (#1-2,#2+1) and (#1-2,#2-1) .. (#1,#2-2);
	\draw[thick] (#1,#2+2) .. controls (#1+2,#2+1) and (#1+2,#2-1) .. (#1,#2-2);

	\fill[white] (#1,#2+2) circle (.1);
	\draw (#1,#2+2) circle (.1);
	\fill (#1,#2-2) circle (.1);

	\cc {#1} {#2}
}

\def\arca#1#2{
	\draw[red] (#1,#2-2) .. controls (#1-2,#2+.75) and (#1+2,#2+.75) .. (#1,#2-2);
}

\def\arcb#1#2{
	\draw[red] (#1,#2+2) .. controls (#1-2,#2-.75) and (#1+2,#2-.75) .. (#1,#2+2);
}

\def\arcca#1#2{
	\draw[red] (#1,#2-2) .. controls (#1-3,#2+1.5) and (#1+3,#2+1.5) .. (#1,#2-2);
}

\def\arccb#1#2{
	\draw[red] (#1,#2+2) .. controls (#1-3,#2-1.5) and (#1+3,#2-1.5) .. (#1,#2+2);
}

\def\arcc#1#2{
	\draw[red] (#1,#2-2) -- (#1,#2+2);
}

\def\arcd#1#2{
	\draw[red] (#1,#2+.25) .. controls (#1+.75,#2+1) and (#1+.75,#2-1) .. (#1,#2-.25);
}

\def\mobdeuxcover#1#2{

		\draw[thick] (#1,#2+1) -- (#1+6,#2+1);
		\draw[thick] (#1,#2-1) -- (#1+6,#2-1);
		\fill[white] (#1,#2+1) circle (.1);
		\draw (#1,#2+1) circle (.1);
		\fill[white] (3,#2-1) circle (.1);
		\draw (#1+3,#2-1) circle (.1);
		\fill[white] (#1+6,#2+1) circle (.1);
		\draw (#1+6,#2+1) circle (.1);
		\fill (#1,#2-1) circle (.1);
		\fill (#1+3,#2+1) circle (.1);
		\fill (#1+6,#2-1) circle (.1);
}

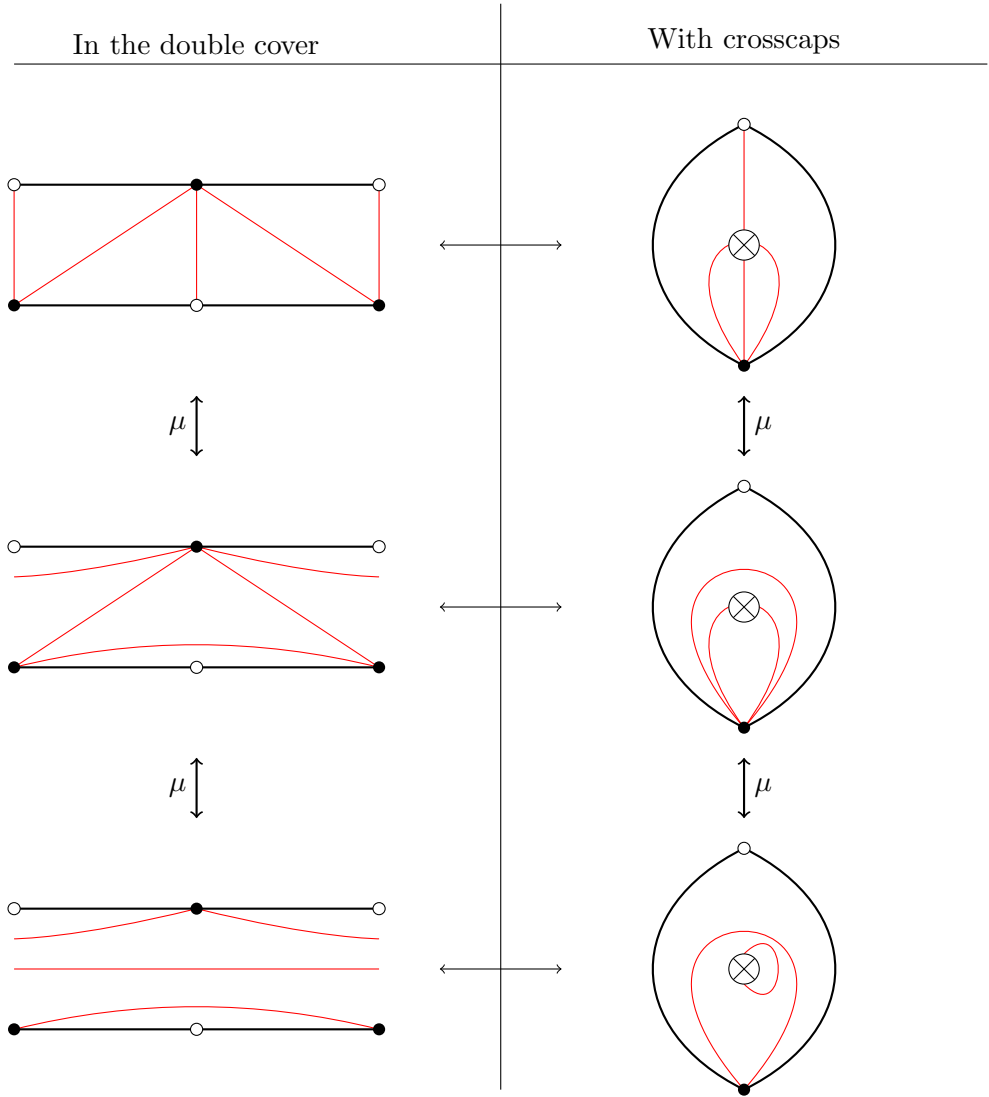
\begin{figure}
	\begin{center}
		\begin{tikzpicture}[scale = .8]
			
			\draw (8,16) -- (8,-2);
			\draw (0,15) -- (16,15);
			\fill (3,15) node [above] {In the double cover};

			\fill (12,15) node [above] {With crosscaps};

			\draw[<->] (7,12) -- (9,12);
			\draw[<->] (7,6) -- (9,6);
			\draw[<->] (7,0) -- (9,0);

			\draw[red] (0,13) -- (0,11);
			\draw[red] (3,13) -- (3,11);
			\draw[red] (6,13) -- (6,11);

			\draw[red] (0,11) -- (3,13);
			\draw[red] (3,13) -- (6,11);
			\mobdeuxcover {0} {12}

			\draw[thick,<->] (3,9.5) -- (3,8.5);
			\fill (3,9) node [left] {$\mu$};

			\arcc {12} {12}
			\arca {12} {12}
 			\mobdeuxbw {12} {12}

			\draw[thick,<->] (12,9.5) -- (12,8.5);
			\fill (12,9) node [right] {$\mu$};

			\draw[red] (0,5) .. controls (2,5.5) and (4,5.5) .. (6,5);
			\draw[red] (0,6.5) .. controls (0,6.5) and (1,6.5) .. (3,7);
			\draw[red] (3,7) .. controls (5,6.5) and (6,6.5) .. (6,6.5);

			\draw[red] (0,5) -- (3,7);
			\draw[red] (3,7) -- (6,5);
			\mobdeuxcover {0} {6}

			\draw[thick,<->] (3,3.5) -- (3,2.5);
			\fill (3,3) node [left] {$\mu$};

			\arcca {12}{6} 
			\arca {12} 6
			\mobdeuxbw {12} {6}

			\draw[thick,<->] (12,3.5) -- (12,2.5);
			\fill (12,3) node [right] {$\mu$};

			\draw[red] (0,-1) .. controls (2,-.5) and (4,-.5) .. (6,-1);
			\draw[red] (0,.5) .. controls (0,.5) and (1,.5) .. (3,1);
			\draw[red] (3,1) .. controls (5,.5) and (6,.5) .. (6,.5);

			\draw[red] (0,0) -- (6,0);
			\mobdeuxcover {0} {0}

			\arcca {12} 0 
			\arcd {12} 0 
			\mobdeuxbw {12} 0 
		\end{tikzpicture}
	\end{center}
	\caption{A glossary for pictures using crosscaps in terms of double covers.}\label{Fig:glossary}
\end{figure}

		The lift $\overline T = \ens{\overline t_1, \ldots, \overline t_n,\tau\overline t_1, \ldots, \tau\overline t_n}$ of a triangulation $T=\ens{t_1, \ldots, t_n}$ of $\SM$ is a collection of compatible arcs, which is maximal due to Euler characteristic considerations. So $\overline T$ provides a triangulation of $\bSM$ which is invariant under the $\Z_2$-action. 

We give in Figure \ref{Fig:glossary}, a glossary between quasi-triangulations of the M\"obius strip with two marked points, and the corresponding arcs or curves in its double cover, which is an annulus with two marked points on each boundary component.

		\begin{rmq}
			Note that a quasi-triangulation of $\SM$ which is not a triangulation does not lift to a triangulation of $\bSM$. Indeed, a one-sided curve in $\SM$ lifts to a non-contractible closed curve in $\bSM$ so that it is not an arc and thus it is not part of a triangulation of $\bSM$. The last quasi-triangulation in Figure \ref{Fig:glossary} gives such an example.
		\end{rmq}

		\begin{lem}\label{lem:mutorbitale}
			Let $\Sigma=(T,\x)$ be a seed associated with $\SM$ in $\mathcal F$ and let $t \in T$ be a mutable arc with respect to $T$. Then 
			$$\mu_ {\overline t} \circ \mu_{\tau \overline t}(\overline \Sigma) = \mu_ {\tau\overline t} \circ \mu_{\overline t}(\overline \Sigma) = \overline{\mu_t(\Sigma)}.$$
		\end{lem}
		\begin{proof}
			Since the arc is mutable with respect to $t$, there exist $a,b,c,d \in T \sqcup \B\SM$ distinct from $t$ such that in $\SM$ we have the following situation~:
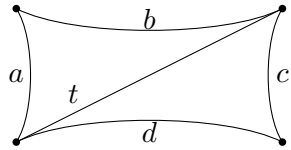
			\begin{center}
				\begin{tikzpicture}[scale = .5]
					\draw (.5,.23) .. controls (2,1) and (6,1) .. (7.5,.23);

					\draw (.5,3.77) .. controls (2,3) and (6,3) .. (7.5,3.77);

					\draw (.5,.23) .. controls (1,1) and (1,3) .. (.5,3.77);	
					\draw (7.5,.23) .. controls (7,1) and (7,3) .. (7.5,3.77);	

					\draw (.5,.23) -- (7.5,3.77);

					\fill (.5,.23) circle (.1);
					\fill (.5,3.77) circle (.1);
					\fill (7.5,3.77) circle (.1);
					\fill (7.5,.23) circle (.1);

					\fill (.5,2) node {$a$};
					\fill (7.5,2) node {$c$};
					\fill (4,3.5) node {$b$};
					\fill (4,.5) node {$d$};

					\fill (2,1) node [above] {$t$};
				\end{tikzpicture}
			\end{center}
			Therefore $\mu_t(\Sigma)$ is given by the triangulation
			
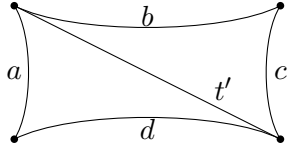
\begin{center}
				\begin{tikzpicture}[scale = .5]
					\draw (.5,.23) .. controls (2,1) and (6,1) .. (7.5,.23);
					\draw (.5,3.77) .. controls (2,3) and (6,3) .. (7.5,3.77);

					\draw (.5,.23) .. controls (1,1) and (1,3) .. (.5,3.77);	
					\draw (7.5,.23) .. controls (7,1) and (7,3) .. (7.5,3.77);	

					\draw (.5,3.77) -- (7.5,.23);

					\fill (.5,.23) circle (.1);
					\fill (.5,3.77) circle (.1);
					\fill (7.5,3.77) circle (.1);
					\fill (7.5,.23) circle (.1);

					\fill (.5,2) node {$a$};
					\fill (7.5,2) node {$c$};
					\fill (4,3.5) node {$b$};
					\fill (4,.5) node {$d$};

					\fill (6,1) node [above] {$t'$};
				\end{tikzpicture}
			\end{center}
			and in the cluster $\x$, all the variables are preserved except $x_t$ which is replaced by
			$$x_t'= \frac{x_ax_c+x_bx_d}{x_t}.$$

			Let $\overline \Sigma=(\overline T, \overline {\x})$ be the lift of $\Sigma$. Then, in $\bSM$, we have the following two distinct quadrilaterals with where all the edges boundaries of the quadrilaterals are distinct from $\overline t$ and $\tau \overline t$~:
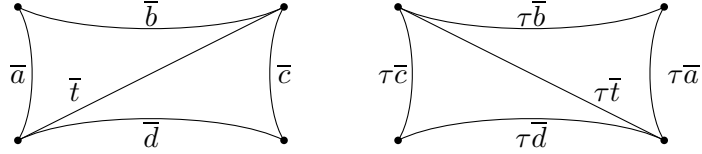
			\begin{center}
				\begin{tikzpicture}[scale = .5]
					\draw (.5,.23) .. controls (2,1) and (6,1) .. (7.5,.23);
					\draw (.5,3.77) .. controls (2,3) and (6,3) .. (7.5,3.77);

					\draw (.5,.23) .. controls (1,1) and (1,3) .. (.5,3.77);	
					\draw (7.5,.23) .. controls (7,1) and (7,3) .. (7.5,3.77);	

					\draw (.5,.23) -- (7.5,3.77);

					\fill (.5,.23) circle (.1);
					\fill (.5,3.77) circle (.1);
					\fill (7.5,3.77) circle (.1);
					\fill (7.5,.23) circle (.1);

					\fill (.5,2) node {$\overline a$};
					\fill (7.5,2) node {$\overline c$};
					\fill (4,3.6) node {$\overline b$};
					\fill (4,.4) node {$\overline d$};

					\fill (2,1) node [above] {$\overline t$};

					\draw (10.5,.23) .. controls (12,1) and (16,1) .. (17.5,.23);
					\draw (10.5,3.77) .. controls (12,3) and (16,3) .. (17.5,3.77);

					\draw (10.5,.23) .. controls (11,1) and (11,3) .. (10.5,3.77);	
					\draw (17.5,.23) .. controls (17,1) and (17,3) .. (17.5,3.77);	

					\draw (10.5,3.77) -- (17.5,.23);

					\fill (10.5,.23) circle (.1);
					\fill (10.5,3.77) circle (.1);
					\fill (17.5,3.77) circle (.1);
					\fill (17.5,.23) circle (.1);

					\fill (11,2) node [left] {$\tau \overline c$};
					\fill (18,2) node {$\tau \overline a$};
					\fill (14,3.6) node {$\tau \overline b$};
					\fill (14,.4) node {$\tau \overline d$};

					\fill (16,1) node [above] {$\tau \overline t$};
				\end{tikzpicture}
			\end{center}
			Therefore, the triangulations in the seeds $\mu_ {\overline t} \circ \mu_{\tau \overline t}(\overline \Sigma)$ and $\mu_ {\tau \overline t} \circ \mu_{\overline t}(\overline \Sigma)$ are given by~:
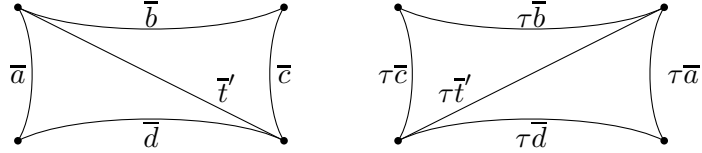
			\begin{center}
				\begin{tikzpicture}[scale = .5]
					\draw (.5,.23) .. controls (2,1) and (6,1) .. (7.5,.23);
					\draw (.5,3.77) .. controls (2,3) and (6,3) .. (7.5,3.77);

					\draw (.5,.23) .. controls (1,1) and (1,3) .. (.5,3.77);	
					\draw (7.5,.23) .. controls (7,1) and (7,3) .. (7.5,3.77);	

					\draw (.5,3.77) -- (7.5,.23);

					\fill (.5,.23) circle (.1);
					\fill (.5,3.77) circle (.1);
					\fill (7.5,3.77) circle (.1);
					\fill (7.5,.23) circle (.1);

					\fill (.5,2) node {$\overline a$};
					\fill (7.5,2) node {$\overline c$};
					\fill (4,3.6) node {$\overline b$};
					\fill (4,.4) node {$\overline d$};

					\fill (6,1) node [above] {$\overline t'$};

					\draw (10.5,.23) .. controls (12,1) and (16,1) .. (17.5,.23);
					\draw (10.5,3.77) .. controls (12,3) and (16,3) .. (17.5,3.77);

					\draw (10.5,.23) .. controls (11,1) and (11,3) .. (10.5,3.77);	
					\draw (17.5,.23) .. controls (17,1) and (17,3) .. (17.5,3.77);	

					\draw (10.5,.23) -- (17.5,3.77);

					\fill (10.5,.23) circle (.1);
					\fill (10.5,3.77) circle (.1);
					\fill (17.5,3.77) circle (.1);
					\fill (17.5,.23) circle (.1);

					\fill (11,2) node [left] {$\tau \overline c$};
					\fill (18,2) node {$\tau \overline a$};
					\fill (14,3.6) node {$\tau \overline b$};
					\fill (14,.4) node {$\tau \overline d$};

					\fill (12,1) node [above] {$\tau \overline t'$};
				\end{tikzpicture}
			\end{center}
			and the corresponding clusters are obtained from $\overline \x$ by replacing respectively $x_{\overline t}$ and $x_{\tau \overline t'}$ by 
			$$x_{\overline t'} = \frac{x_{\overline a}x_{\overline c}+x_{\overline b}x_{\overline d}}{x_{\overline t}} \text{ and } x_{\tau \overline t'} = \frac{x_{\tau \overline a}x_{\tau \overline c}+x_{\tau \overline b}x_{\tau \overline d}}{x_{\tau \overline t}}.$$
			Therefore, $\mu_ {\overline t} \circ \mu_{\tau \overline t}(\overline \Sigma)=\mu_ {\tau \overline t} \circ \mu_{\overline t}(\overline \Sigma)$ is the lift of the seed $\mu_t(\Sigma)$, which proves the lemma.
		\end{proof}

		\begin{rmq}
			Note that Lemma \ref{lem:mutorbitale} does not hold if $t$ is not mutable with respect to $T$. For instance, if we consider the M\"obius strip $\mathcal M_1$ with one marked point and the following triangulation $T$~:
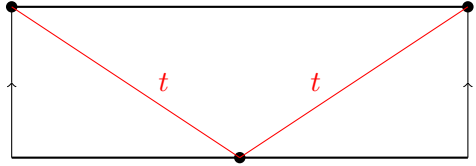
			\begin{center}
				\begin{tikzpicture}[scale = 1]
					\draw[thick] (0,2) -- (6,2);
					\draw[thick] (0,0) -- (6,0);
					
					\draw[->] (0,0) -- (0,1);
					\draw (0,1) -- (0,2);

					\draw[->] (6,0) -- (6,1);
					\draw (6,1) -- (6,2);

					\fill (0,2) circle (.075);
					\fill (6,2) circle (.075);
					\fill (3,0) circle (.075);

					\draw[red] (0,2) -- (3,0);
					\draw[red] (6,2) -- (3,0);

					\fill[red] (2,1) node {$t$};
					\fill[red] (4,1) node {$t$};
				\end{tikzpicture}
			\end{center}
			Then $t$ is not mutable with respect to $T$ and the quasi-mutation gives the following quasi-triangulation.
			
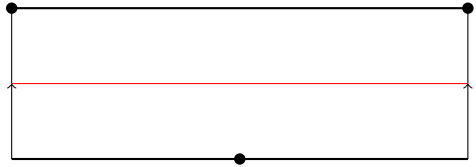
\begin{center}
				\begin{tikzpicture}[scale = 1]
					\draw[thick] (0,2) -- (6,2);
					\draw[thick] (0,0) -- (6,0);
					
					\draw[->] (0,0) -- (0,1);
					\draw (0,1) -- (0,2);

					\draw[->] (6,0) -- (6,1);
					\draw (6,1) -- (6,2);

					\fill (0,2) circle (.075);
					\fill (6,2) circle (.075);
					\fill (3,0) circle (.075);

					\draw[red] (0,1) -- (6,1);
				\end{tikzpicture}
			\end{center}

			In the double cover, which is the annulus $C_{1,1}$ with one marked point on each boundary component, the lift of $T$ is the following triangulation~:
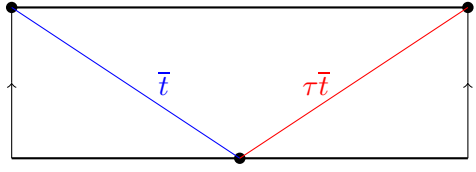
			\begin{center}
				\begin{tikzpicture}[scale = 1]
					\draw[thick] (0,2) -- (6,2);
					\draw[thick] (0,0) -- (6,0);
					
					\draw[->] (0,0) -- (0,1);
					\draw (0,1) -- (0,2);

					\draw[->] (6,0) -- (6,1);
					\draw (6,1) -- (6,2);

					\fill (0,2) circle (.075);
					\fill (6,2) circle (.075);
					\fill (3,0) circle (.075);

					\draw[blue] (0,2) -- (3,0);
					\draw[red] (6,2) -- (3,0);

					\fill[blue] (2,1) node {$\overline t$};
					\fill[red] (4,1) node {$\tau \overline t$};
				\end{tikzpicture}
			\end{center}

			The sequence of mutations $\mu_{\tau \overline t} \circ \mu_ {\overline t}$ gives the following triangulation of the double cover~:
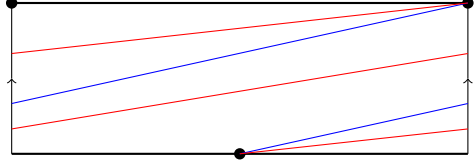
			\begin{center}
				\begin{tikzpicture}[scale = 1]
					\draw[thick] (0,2) -- (6,2);
					\draw[thick] (0,0) -- (6,0);
					
					\draw[->] (0,0) -- (0,1);
					\draw (0,1) -- (0,2);

					\draw[->] (6,0) -- (6,1);
					\draw (6,1) -- (6,2);

					\fill (0,2) circle (.075);
					\fill (6,2) circle (.075);
					\fill (3,0) circle (.075);

					\draw[blue] (6,2) -- (0,.667);
					\draw[blue] (3,0) -- (6,.667);

					\draw[red] (6,2) -- (0,1.33);
					\draw[red] (6,1.33) -- (0,.33);
					\draw[red] (3,0) -- (6,.33);
				\end{tikzpicture}
			\end{center}
			Whereas the sequence $\mu_{\overline t}\circ \mu_{\tau \overline t}$ gives the following triangulation of the double cover~:
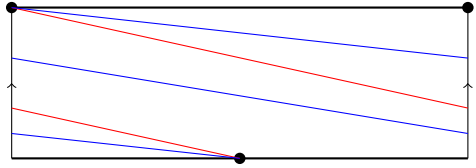
			\begin{center}
				\begin{tikzpicture}[scale = 1]
					\draw[thick] (0,2) -- (6,2);
					\draw[thick] (0,0) -- (6,0);
					
					\draw[->] (0,0) -- (0,1);
					\draw (0,1) -- (0,2);

					\draw[->] (6,0) -- (6,1);
					\draw (6,1) -- (6,2);

					\fill (0,2) circle (.075);
					\fill (6,2) circle (.075);
					\fill (3,0) circle (.075);

					\draw[red] (0,2) -- (6,.667);
					\draw[red] (3,0) -- (0,.667);

					\draw[blue] (0,2) -- (6,1.33);
					\draw[blue] (0,1.33) -- (6,.33);
					\draw[blue] (3,0) -- (0,.33);
				\end{tikzpicture}
			\end{center}
			Therefore the mutations $\mu_{\overline t}$ and $\mu_{\tau \overline t}$ do not commute and moreover, their respective products do not give rise to lifts of quasi-triangulations of the M\"obius strip $\mathcal M_1$.
		\end{rmq}

	\subsection{Quotient map}
		 Given a seed $\Sigma = (T,\x)$ associated with $\SM$, we denote by $\overline \Sigma$ the seed $(\overline T,\overline \x)$ corresponding to a lift of $T$ in $\bSM$. The group $\Z_2$ acts naturally on the ambient field $\overline{\mathcal F}$ of $\mathcal A_{\bSM}$ by $\tau x_{\overline t} = x_{\tau\overline t}$ for any $t \in T \sqcup \B\SM$ and we consider the $\Z_2$-invariant ring epimorphism~:
		$$\pi:\left\{\begin{array}{rcl}
			\overline{\mathcal F} & \longrightarrow & \mathcal F \\
			x_{\overline t} & \longmapsto & x_{t} \text{ for any }t \in T \sqcup \B\SM,\\
			x_{\tau\overline t} & \longmapsto & x_{t}\text{ for any }t \in T \sqcup \B\SM.
		\end{array}\right.$$
		
		\begin{lem}\label{lem:pix}
			With the above notations, $\pi(\overline x_{\overline t}) = \pi(\overline x_{\tau \overline t}) = x_t$ for any $t \in \A\SM \sqcup \B\SM$.
		\end{lem}
		\begin{proof}
			The exchange graph $\E\SM$ is connected and all the mutations in this exchange graph are done in the direction of mutable arcs, see Corollary \ref{corol:caracqmut}. Let $t$ be an element in $\A\SM \sqcup \B\SM$ and let $T'$ be a triangulation of $\SM$ containing $t$ and such that the distance $d(T,T')$ between $T$ and $T'$ in $\E\SM$ is minimal. We prove the result by induction on this minimal distance. There is a triangulation $T''$ such that $t$ is mutable in $T''$ with $T'=\mu_t(T'')$ and $d(T,T'')<d(T,T')$. By induction hypothesis, the result holds for any arc in the triangulation $T''$. Since $t$ is mutable in $T''$, we are in the situation of Lemma \ref{lem:mutorbitale} and, with the same notations as in the proof of this lemma, we can apply Ptolemy relations in both $\SM$ and $\bSM$ and we get~:
			$$x_t x_{t'} = x_ax_c+x_bx_d \text{ and } x_{\overline t} x_{\overline {t'}} = x_{\overline a}x_{\overline c}+x_{\overline b}x_{\overline d}.$$
			Therefore, we get 
			$$\pi(x_{\overline t})x_{t'}=\pi(x_{\overline t})\pi(x_{\overline {t'}}) = \pi(x_{\overline a})\pi(x_{\overline c})+\pi(x_{\overline b})\pi(x_{\overline d}) = x_ax_c+x_bx_d$$
			and thus $\pi(x_{\overline t})=x_t$, which proves the lemma.
		\end{proof}

	\subsection{Exchange graphs from double covers}
		In this section, we show that for a non-orientable surface $\SM$, we can recover the exchange graph of $\SM$ in terms of $\Z_2$-invariant points in the exchange graph $\E\bSM$.

		Let $\A^{\Z_2}\bSM$ denote the set of $\Z_2$-orbits of arcs $a$ in $\bSM$ such that $a$ and $\tau a$ have no intersections. Two elements in $\A^{\Z_2}\bSM$ are called compatible if the union of the corresponding two $\Z_2$-orbits consists of pairwise compatible arcs in $\bSM$. We denote by $\Delta^{\Z_2}\bSM$ the simplicial complex on the ground set $\A^{\Z_2}\bSM$ defined as the clique complex for the compatibility relation. The vertices in $\Delta^{\Z_2}\bSM$ are the $\Z_2$-orbits of arcs $a$ in $\bSM$ such that $a$ and $\tau a$ have no intersection and the maximal simplices are the $\Z_2$-invariant triangulations of $\bSM$.

		We denote by $\E^{\Z_2}\bSM$ the dual graph of $\Delta^{\Z_2}\bSM$. The vertices in $\E^{\Z_2}\bSM$ are the $\Z_2$-invariant triangulations and two $\Z_2$-invariant triangulations $T$ and $T'$ of $\bSM$ are related by an edge in $\E^{\Z_2}\bSM$ if and only if there exists $a \in T$ and $a' \in T'$ such that $T \setminus \ens{a,\tau a} = T' \setminus \ens{a',\tau a'}$.

		\begin{prop}
			$$\E\SM \simeq \E^{\Z_2}\bSM.$$
		\end{prop}
		\begin{proof}
			The map sending a triangulation $T$ of $\SM$ to its lift $\overline T$ in $\bSM$ induces a bijection from the set of triangulations in $\SM$ to the set of $\Z_2$-invariant triangulations of $\bSM$. Thus we need to prove that two triangulations $T$ and $T'$ of $\SM$ differ by a single arc if and only if their lifts $\overline T$ and $\overline{T'}$ differ by a single $\Z_2$-orbit of arcs.

			Let $T,T'$ be triangulations of $\SM$ and assume that there is some collection $T_0$ of arcs in $\SM$ and $a,b \in \A\SM$ such that $T = T_0 \sqcup \ens a$ and $T' = T_0 \sqcup \ens b$. Then $\overline T = \overline{T_0} \sqcup \ens{a,\tau a}$ and $\overline{T'} = \overline{T_0} \sqcup \ens{b,\tau b}$. Since both $\overline T$ and $\overline{T'}$ are $\Z_2$-invariant triangulations, it follows that $a\cap \tau a = \emptyset$ and $b\cap \tau b = \emptyset$. Thus, $\overline T$ and $\overline{T'}$ are related by an edge in $\E^{\Z_2}\bSM$.

			Conversely, let $\overline T$ and $\overline{T'}$ be two distinct $\Z_2$-invariant triangulations of $\bSM$ lifting respectively the triangulations $T$ and $T'$ in $\SM$. Assume that we can write $\overline T = \overline{T_0} \sqcup \ens{a,\tau a}$ and $\overline{T'} = \overline{T_0} \sqcup \ens{b,\tau b}$. Note that $a$ and $\tau a$ are not the internal arcs of an anti-self-folded otherwise we would necessarily have $\ens{a,\tau a} = \ens{b,\tau b}$ and $\overline T=\overline{T'}$. Thus, it follows from Corollary \ref{corol:caracqmut} that $a$ is mutable with respect to $T$ and thus, it follows from Lemma \ref{lem:mutorbitale} that $\overline {T'} = \overline{\mu_a(T)}$ so that $T' = \mu_a(T)$ and $T'$ and $T$ are joined by an edge in $\E\SM$.
		\end{proof}

	\subsection{A combinatorial rule for mutations of triangulations}
		Considering exchange matrices in $\mathcal A_{\bSM}$ associated with lifts in $\bSM$ of triangulations in $\SM$, it is possible to define a combinatorial mutation rule for mutations of cluster variables in $\mathcal A_{\SM}$. Nevertheless, it appears that this method does not generalise to quasi-mutations which are not mutations.

		Let $\Sigma=(T,\x)$ be a seed associated with $\SM$ in $\mathcal F$. We fix an arbitrary orientation of the double cover $\bSM$ and we denote by $\overline B$ the matrix associated with the lift $\overline T$ of the triangulation $T$ in $\bSM$. We recall that the entries of this matrix are indexed by the lifts of arcs of $T$ and for any two arcs $v,w$ in $T$, the entry corresponding to the lifts $\overline v$ and $\overline w$ is defined as the difference 
		$$[\overline B]_{\overline v,\overline w} = n_{\overline T}(\overline v,\overline w) - n_{\overline T}(\overline w,\overline v)$$ where, for any arcs $a$ and $b$, the number $n_{\overline T}(a,b)$ is given by number of triangles in $\overline T$ bordered by $a$ and $b$ in such a way that the oriented angle formed by $a$ and $b$ in this triangle is positive, see \cite[Section 4]{FST:surfaces}.

		For any arc $v \in T$, we set 
		$$b^+_{tv} = \max([\overline B]_{\overline t,\overline v},0)+ \max([\overline B]_{\overline t,\tau \overline v},0) = \max([\overline B]_{\tau \overline t,\overline v},0)+ \max([\overline B]_{\tau \overline t,\tau \overline v},0)$$
		$$b^-_{tv} = \min([\overline B]_{\overline t,\overline v},0)+ \min([\overline B]_{\overline t,\tau \overline v},0) = \min([\overline B]_{\tau \overline t,\overline v},0)+ \min([\overline B]_{\tau \overline t,\tau \overline v},0).$$
		Note that even if the matrix $\overline B$ depends on the choice of the orientation of $\bSM$, the pair $\ens{b^+_{tv},b^-_{tv}}$ is independent on this choice.

		\begin{prop}\label{prop:mutationarc}
			Let $T$ be a triangulation of $\SM$, let $t \in T$ be mutable with respect to $T$ and let $t'$ denote its flip with respect to $T$. Then
			$$x_t x_{t'} = \prod_{v \in T} x_v^{b^+_{tv}} + \prod_{v \in T} x_v^{-b^-_{tv}}$$
			and the matrix associated with $\mu_t(T)$ is $\mu_{\overline t} \circ \mu_{\tau \overline t}(\overline B) = \mu_{\tau \overline t} \circ \mu_{\overline t}(\overline B)$.
		\end{prop}
		\begin{proof}
			This is a direct consequence of Lemmas \ref{lem:mutorbitale} and \ref{lem:pix} and of the definition of the exchange relations for a cluster algebra of geometric type. 
		\end{proof}
		

\subsection{Laurent Phenomenon}

Using the previous description of exchange relations corresponding to mutations in the cluster algebra of the double cover, we can now prove that quasi-cluster algebras satisfy the Laurent property.  

\begin{theorem}\label{thm:laurent} In a quasi-cluster algebra $\mathcal A_{\SM}$, any quasi-cluster variable is expressed in terms of any quasi-cluster as a Laurent polynomial with $\Z\P$ coefficients.
\end{theorem}

\begin{proof} 
	Let $v$ be a quasi-cluster variable in $\mathcal A_{\SM}$ and $\bf{x}= \{ x_t \mid t \in T\}$ a quasi-cluster associated to a quasi-triangulation $T$. 
	
	It is sufficient to prove the theorem in the case where $T$ is a triangulation. Indeed, if $T$ is a quasi-triangulation containing a one-sided curve $t$, there exists a unique arc $c \in T$, such that $c$ is the boundary of an embedded M\"obius band and $t$ is the unique one-sided curve inside it. Now let $T'$ denote the triangulation obtained by mutation of $T$ in the direction $t$, and let $\bf{x}' = \bf{x}\cup \{ x_a \} \setminus \{ x_t \}$ the new cluster where $a \in T'$ is the new arc in this triangulation. The relation between the variables is given by Definition \ref{defi:qmut} as $x_t x_a = x_c$. From this relation, it is clear that any Laurent polynomial in terms of the cluster $\bf{x}'$ can be expressed as a Laurent polynomial in terms of the cluster $\bf{x}$.
	
	We now assume that $T$ is a triangulation. First, if $v$ is a quasi-cluster variable associated to an arc, we can find a sequence of mutations, without quasi-mutations, from the initial cluster $\bf{x}$ to another cluster containing the cluster variable $v$. Then the exchange relations given in Proposition \ref{prop:mutationarc} prove that each mutation in $\mathcal A_{\SM}$ can be thought of as a composition of two cluster mutations in the cluster algebra $\mathcal A_{\bSM}$ of the double cover. Hence using the Laurent property of the cluster algebra $\mathcal A_{\bSM}$, this proves that we can express the variable $v$ as a Laurent polynomial in terms of the quasi-cluster $\bf{x}$.  
	
	The only remaining case occurs when $v$ is a quasi-cluster variable associated to a one-sided curve $d$. As the triangulation $T$ cuts $\S$ into orientable triangles and annuli, there exists an arc $t \in T$ such that $t$ and $d$ intersect at least once. Take one of these intersection and use the results of section \ref{subsub:onearconecurve} to show that there exist two arcs $y$ and $z$, not necessarily simple, with lambda lengths satisfying the following relation :
	$$x_t x_v = x_y + x_z$$
Now it suffices to prove that $x_y$ and $x_z$ can be expressed as Laurent polynomials in the cluster $\bf{x}$.  So lift the triangulation $T$ in the universal cover as a geodesic triangulation $\widetilde{T}$ of the hyperbolic plane. Choose any two lifts of $y$ and $z$ as geodesics in the hyperbolic plane starting at the same ideal point. These geodesic encounter a finite number of triangles and we can define an ideal polygon $\mathcal{P}$ containing all these triangles. The polygon $\mathcal{P}$ is triangulated by lifts of arcs of the triangulation $T$, so the lambda-lengths corresponding to these arcs correspond to lambda-length of arcs in the initial triangulation $T$. So in the cluster algebra $\mathcal A_{\mathcal P}$ defined by the triangulation of this ideal polygon, we can identify the variables in the initial seed $\bf{p}$ with elements of the initial cluster $\bf{x}$ of the quasi-cluster algebra $\mathcal A_{\SM}$. Now, in $\mathcal A_{\mathcal P}$, the variable $x_y$ and $x_z$ are cluster variables. Hence $x_y$ and $x_z$ can be expressed as Laurent polynomials in terms of the variables in the initial seed $\bf{p}$ which all correspond to variables in the cluster $\bf{x}$, which ends the proof of the theorem.
	\end{proof}

	\begin{rmq}
		Since the first draft of this paper, a new kind of algebra called \emph{Laurent Phenomenon Algebras} have been defined by Lam and Pylyavskyy \cite{LP:LPalgebras}, to generalize cluster algebras. Their setting is slightly different, but it seems that quasi-cluster algebras can be seen as a particular instance of these more general objects. 
	\end{rmq}


\section{Finite type classification}
	Cluster algebras of finite type were defined in \cite{cluster2} as cluster algebras with finitely many cluster variables and are classified by Dynkin diagrams. For cluster algebras coming from surfaces, the cluster algebras of finite type are those associated either with a disc with at least four marked points on the boundary (which correspond to Dynkin type $A$) or those associated with a disc with at least four marked points on the boundary and with one puncture (which correspond to Dynkin type $D$). In this section, we provide a similar classification for quasi-cluster algebras.

	\begin{defi}
		A quasi-cluster algebra $\mathcal A_{\SM}$ is called of \emph{finite type} if it has finitely many quasi-cluster variables or, equivalently, if the set $\qA\SM$ is finite.
	\end{defi}

	\begin{theorem}\label{theorem:classification}
		A quasi-cluster algebra $\mathcal A_{\SM}$ is of finite type if and only if $\SM$ is one of the following marked surfaces~:
		\begin{enumerate}
			\item a disc with at least four marked points on the boundary,
			\item a M\"obius strip with at least one marked point on the boundary. 
		\end{enumerate}
	\end{theorem}
	\begin{proof}
		If $\S$ is orientable, then the classification of finite type is classical. So assume that $\S$ is non-orientable. If $\S$ has two or more boundary components, then the boundary twist along one of the boundary component, which is the homeomorphism that sends the boundary to itself after a $2 \pi$ rotation, generates an infinite cyclic subgroup of the mapping class group. The orbit of a simple arc joining this boundary component to another one is infinite, and hence we have an infinite number of quasi-arcs in $\S$.
 
		If $\S$ is of non-orientable genus greater than two, then $\S$ contains a one-holed Klein bottle $K$. It is known that there exists an infinite number of one-sided simple closed curves in $K$, all in the orbit of a single element under by the action of the Dehn twist along the unique non-trivial two-sided simple closed curve in $K$. Hence we get an infinite number of quasi-arcs in $\S$.
	\end{proof}

	\begin{rmq}
		For cluster algebras of finite type, it is known that the number of cluster variables is given by the number of almost positive roots of the corresponding Dynkin diagram, see \cite{cluster2}. For the the M\"obius strip $\mathcal M_n$ with $n \geq 1$ marked points on the boundary, an easy calculation shows that the number of quasi-arcs, and thus of cluster variables in $\mathcal A_{\mathcal M_n}$ is given by 
		$$|\qA(\mathcal M_n)| = \frac{3n^2-n+2}{2},$$
		whereas the number of arcs is given by 
		$$\displaystyle |\A(\mathcal M_n)| = |\qA(\mathcal M_n)|-1 = \frac{n(3n-1)}{2}.$$

		Note in particular that the number of quasi-cluster variables in $\A(\mathcal M_n)$ does not coincide with the number of cluster variables in any cluster algebra of finite type. In particular, the quasi-cluster algebra structure on $\A(\mathcal M_n)$ is not a cluster algebra structure for any $n \geq 1$.
	\end{rmq}

	\subsection{Linear bases in quasi-cluster algebras of finite type}
		Throughout this section $\SM$ denotes a non-oriented unpunctured marked surface of rank $n \geq 1$.

		\begin{defi}
			Let $\mathcal A_{\SM}$ be a quasi-cluster algebra. A \emph{quasi-cluster monomial} (resp. a \emph{cluster monomial}) in $\mathcal A_{\SM}$ is a monomial in quasi-cluster variables belonging all to the same quasi-cluster (resp. cluster). 
		\end{defi}

		We denote by $\qW\SM$ the set of \emph{weighted quasi-triangulations}~:
		$$\qW\SM = \ens{(t_i,n_i)_{1 \leq i \leq n} \ | \ \bigcup_{i=1}^n t_i \in \qT\SM, \ n_i \geq 0},$$
		and the set of \emph{weighted triangulations} by
		$$\W\SM = \ens{(t_i,n_i)_{1 \leq i \leq n} \ | \ \bigcup_{i=1}^n t_i \in \T\SM, \ n_i \geq 0},$$
		and for any $\alpha \in \qW\SM$, we set 
		$$x_\alpha = \prod_{i=1}^n x_{t_i}^{n_i}.$$
		
		Thus, the set of quasi-cluster monomials in $\mathcal A_{\SM}$ is 
		$$\mathfrak M_{\SM} = \ens{x_\alpha \ | \ \alpha \in \qW\SM}.$$

		The quasi-cluster algebra $\mathcal A_{\SM}$ is naturally endowed with a structure of module over its ground ring $\Z\P$ and a \emph{$\ZP$-linear basis} in $\mathcal A_{\SM}$ is a free generating set of $\mathcal A_{\SM}$ for this structure.

		In \cite{CK1}, Caldero and Keller proved that the set of cluster monomials form a $\Z$-linear basis in any coefficient-free cluster algebra of finite type (in the sense of \cite{cluster2}). Here we generalise this result to quasi-cluster algebras of finite type in the above sense. Similar methods recently appeared for cluster algebras associated to arbitrary orientable surfaces, see \cite{MSW:bases}.

		\begin{theorem}\label{theorem:basis}
			Let $\mathcal A_{\SM}$ be a quasi-cluster algebra of finite type. Then the set of quasi-cluster monomials in $\mathcal A_{\SM}$ form a $\Z\P$-linear basis of $\mathcal A_{\SM}$. 
		\end{theorem}
		\begin{proof}
			As a $\Z\P$-module, the quasi-cluster algebra $\mathcal A_{\SM}$ is generated by elements of the form $m= x_\alpha$ where $\alpha$ runs over the set of multigeodesics consisting of quasi-arcs. Thus, in order to prove that quasi-cluster monomials form a generating set over the ground ring $\ZP$, we only have to prove that each such monomial can be written as a $\ZP$-linear combination of quasi-cluster monomials. 

			Let thus $\alpha$ be a multigeodesic consisting of quasi-arcs. Resolving successively the intersections in $\alpha$, we can write $x_\alpha$ as a $\Z$-linear combination of $x_\gamma$ where each $\gamma$ is multigeodesic consisting of pairwise compatible simple geodesics. Let $\gamma$ be one of these multigeodesics. We denote by $\beta$ the subset of $\gamma$ consisting of boundary segments and we set $x_\gamma = x_\beta x_{\gamma'}$. If $\gamma' \in \qW\SM$, we are done. Otherwise, we are necessarily in the non-orientable case and it follows from Theorem \ref{theorem:classification} that $\SM=\mathcal M_n$ for some $n \geq 1$. We denote by $d$ the unique one-sided simple closed curve in $\mathcal M_n$. We thus know that $x_{\gamma'}$ is either of the form $x_{\gamma''}x_d(x_{d^2})^l$ or of the form $x_{\gamma''}(x_{d^2})^l$ for some $l \geq 0$ and some $\gamma'' \in \qW\SM$ compatible with $d$ (or equivalently with $d^2$). Now, it follows from Proposition \ref{prop:d2} that $(x_{d^2})^l$ is a polynomial in $x_d$ with coefficients in $\Z$ so that $x_{\gamma'}$ is a $\Z$-linear combination of elements of the form $x_{\gamma''}x_d^l$ where $l \geq 0$ and $\gamma''$ is compatible with $d$. In other words, $x_\gamma$ is a $\Z\P$-linear combination of elements of the form $x_{\gamma'}$ with $\gamma' \in \qW\SM$.

			We now need to prove that quasi-cluster monomials are linearly independent over the ground ring $\Z\P$. If $\SM$ is orientable, then $\mathcal A_{\SM}$ is a cluster algebra of type $A$ and the result is well-known, see for instance \cite{CK1,MSW:bases}. We thus focus on the case where $\SM$ is non-orientable so that $\SM = \mathcal M_n$ for some $n \geq 1$. The double cover $\bSM$ is therefore the annulus $C_{n,n}$ with $n$ marked points on each boundary component which we endow with an arbitrary orientation. We chose a fundamental domain for the $\Z_2$-action in $C_{n,n}$ and for any $t \in \A(\mathcal M_n) \sqcup \B(\mathcal M_n)$ we denote by $\overline t$ the lift of $t$ in this fundamental domain. And for any multigeodesic $\alpha = \ens{t_1, \ldots, t_m}$ in $\mathcal M_n$, we set $\widetilde \alpha = \ens{\overline{t_1}, \ldots, \overline{t_m}}$ the corresponding multigeodesic in $C_{n,n}$. The $\lambda$-lengths being preserved by the $\Z_2$ action on $C_{n,n}$, we can naturally identify $\mathcal A_{\mathcal M_n}$ with a subalgebra of $\mathcal A_{C_{n,n}}$ via the ring homomorphism $\iota$ sending the cluster variable $x_{t} \in \mathcal A_{\mathcal M_{n}}$ to the cluster variable $x_{\overline t} \in \mathcal A_{C_{n,n}}$. The one-sided curve $d$ in $\mathcal M_n$ has a unique lift in $C_{n,n}$, which we denote by $\overline d$. According to Proposition \ref{prop:d2}, the corresponding $\lambda$-lengths are related by $\lambda(\overline d) = \lambda(d)^2+2$ and we denote by $x_{\overline d}$ the element in the cluster algebra $\mathcal A_{C_{n,n}}$ corresponding to the image of $x_d^2+2$ under $\iota$.

			We have 
			$$\mathfrak M_{\mathcal M_n} \subset \ens{x_d^lx_\alpha\ | \ l \geq 0, \alpha \in \W(\mathcal M_n)}$$
			so that we can fix the decomposition $\mathfrak M_{\mathcal M_n} = \mathfrak M_0 \sqcup \mathfrak M_1$ where
			$$\mathfrak M_0 = \mathfrak M_{\mathcal M_n} \cap \ens{x_d^{2l}x_\alpha \ | \ l \geq 0, \alpha \in \W(\mathcal M_n)}$$ and 
			$$\mathfrak M_1 = \mathfrak M_{\mathcal M_n} \cap \ens{x_d^{2l+1}x_\alpha \ | \ l \geq 0, \alpha \in \W(\mathcal M_n)}.$$ 
			We denote respectively by $M_0$ and $M_1$ the $\Z\P$-modules which these two sets span in $\mathcal A_{\SM}$.

			We first prove that $\mathfrak M_0$ is linearly independent over $\Z\P$. For this, it is enough to prove that its image $\iota(\mathfrak M_0)$ under $\iota$ is linearly independent over the ground ring of $\mathcal A_{C_{n,n}}$. We have 
			\begin{align*}
				\iota(\mathfrak M_0) 
					\subset \ens{\iota(x_d)^{2l} \iota(x_\alpha) \ | \ \alpha \in \W(\mathcal M_n)} \\
					\subset \ens{\iota(x_d)^{2l} x_{\widetilde \alpha} \ | \ \alpha \in \W(\mathcal M_n)} \\
					\subset \ens{\iota(x_d^2)^{l} x_\alpha \ | \ \alpha \in \W(C_{n,n})}.
			\end{align*}
			Now $\ens{\iota(x_d^2+2)^{l} x_{\alpha} \ | \ \alpha \in \W(C_{n,n})}$ is a subset of the \emph{generic basis} of $\mathcal A_{C_{n,n}}$, see \cite{Dupont:BaseAaffine,MSW:bases} and therefore it is linearly independent over the ground ring and so is $\mathfrak M_0$.

			Assume now that there is some vanishing $\Z\P$-linear combination
			$$\sum_{l \geq 0} \sum_{\alpha \in \W(\mathcal M_n)} a_{l,\alpha} x_{\alpha} x_d^{2l+1} =0,$$
			then multiplying by $x_d$, we get 
			$$\sum_{l \geq 0} \sum_{\alpha \in \W(\mathcal M_n)} a_{l,\alpha} x_{\alpha} x_d^{2l+2} =0$$
			and thus each $a_{l,\alpha}$ is zero since $\mathfrak M_0$ is linearly independent over $\Z\P$. Therefore, $\mathfrak M_1$ is also linearly independent over $\Z\P$.

			We now claim that $M_0 \cap M_1 = \ens{0}$. Indeed, assume that there are $\Z\P$-linear combinations such that 
			\begin{equation}\label{eq:linind}
				\sum_{\substack{l \geq 0 \\ \alpha \in \W(\mathcal M_n)}} a_{l,\alpha} x_\alpha x_d^{2l+1}  = \sum_{\substack{k \geq 0 \\ \beta \in \W(\mathcal M_n)}} b_{k,\beta} x_\beta x_d^{2k}
			\end{equation}
			with $a_{l,\alpha}, b_{k,\beta} \in \Z\P$. Then, if we square this identity, the left-hand side is a $\Z\P$-linear combination of products of the form $x_\alpha x_d^{2l+1} x_{\alpha'} x_d^{2l'+1}$ where $\alpha, \alpha' \in \W(\mathcal M_n)$ are compatible with $d$ and where $l,l' \geq 0$. Using Theorem \ref{theorem:resolution}, the product $x_{\alpha}x_{\alpha'}$ can be written as a $\Z\P$-linear combination of $x_{\alpha''}(x_{d^2})^{l''}$ where $\alpha'' \in \W(\mathcal M_n)$ is compatible with $d^2$ and thus with $d$ and where $l'' \geq 0$. Therefore, the square of the left-hand side is a $\Z\P$-linear combination of elements of the form $x_{\alpha}x_d^{2l}$ with $\alpha \in \W(\mathcal M_n)$ compatible with $d$ and $l > 0$. Similarly, the square of the right-hand side is a $\Z\P$-linear combination of elements of the form $x_{\beta}x_d^{2k}$ with $\beta \in \W(\mathcal M_n)$ compatible with $d$ and $k \geq 0$. In particular, the square of each side of \eqref{eq:linind} is a $\Z\P$-linear combination of elements of $\mathfrak M_0$, which is known to be linearly independent over $\Z\P$. Therefore, the coefficients of each $x_{\beta}x_d^{2k}$ with $k=0$ in the square of the right-hand side has to be zero and thus $b_{0,\beta} = 0$ for any $\beta$ occurring in the right-hand side of \eqref{eq:linind}. Therefore, we can divide both sides of \eqref{eq:linind} by the smallest power of $x_d$ arising on one of the two sides and by induction, it follows that $a_{l,\alpha} = b_{k,\beta}=0$ for any $k,l \geq 0$ and $\alpha,\beta \in \W(\mathcal M_n)$. This finishes the proof of the theorem.
		\end{proof}

\section{Integrable systems associated with unpunctured surfaces}
	The aim of this section is to prove that with any unpunctured marked surface, we can naturally associate a family of discrete integrable systems satisfied by $\lambda$-lengths of curves in $\SM$. In the case where the variables corresponding to boundary segments are specialised to 1, these integrable systems provide $SL_2$-tilings of the plane, also called \emph{friezes} in the literature see for instance \cite{ARS:frises}.

	\subsection{AR-quivers for homotopy classes of curves}
		Let $\SM$ be an unpunctured marked surface which is not necessarily oriented. We denote by $\C\SM$ the set of curves in $\SM$ whose both endpoints are in $\M$ considered up to isotopy with respect to $\M$. We define a \emph{one-sided geodesic} as a curve in $\SM$ joining two marked points on the same boundary component such that its concatenation with a boundary component joining its two endpoints reverses the orientation of the surface. 

		We fix two boundary components $\d$ and $\d'$ of $\SM$ which are not necessarily distinct. Let $\mathcal H$ denote a given homotopy class of oriented curves in $\SM$ whose endpoints lie respectively on $\d$ and $\d'$ (but not necessarily on $\M$), where the homotopy do not necessarily preserve $\M$. Finally, denote by $\C_{\mathcal H}$ the set of elements in $\C\SM \sqcup \B\SM$ such that a representative of the isotopy class belongs to $\mathcal H$.

		The boundary components $\d$ and $\d'$ are one-dimensional so that they can both be oriented. If $\SM$ is oriented the boundary components $\d$ and $\d'$ are canonically oriented and we fix orientations $\omega$ and $\omega'$ respectively of $\d$ and $\d'$ which are induced by the orientation of $\SM$. If $\SM$ is not orientable, we fix arbitrary orientations such that $\omega = \omega'$ if $\d = \d'$. 

		We say that the orientations $\omega$ and $\omega'$ of $\d$ and $\d'$ are \emph{compatible with respect to $\mathcal H$} if $\mathcal H$ does not contain any one-sided geodesics. We say that $\omega$ and $\omega'$ are \emph{incompatible with respect to $\mathcal H$} otherwise and in this latter case, $\mathcal H$ consists only of one-sided geodesics. Note that if $\SM$ is oriented then the orientations of $\omega$ and $\omega'$ are always compatible with respect to $\mathcal H$.

		Let $a \in \C_{\mathcal H}$, that is a continuous map $a: [0,1] \fl \S$ such that $a(0) \in \d \cap \M$ and $a(1) \in \d' \cap \M$ or $a(0) \in \d' \cap \M$ and $a(1) \in \d \cap \M$. We define $\Sigma_0 a$ as the element of $\C_{\mathcal H}$ obtained by concatenating the boundary segment joining $a(0)$ to the next marked point along the orientation of the boundary, with the curve $a$. If $\omega$ and $\omega'$ are compatible (or incompatible, respectively), we define $\Sigma_1a$ to be the curve obtained by concatenating $a$ with the boundary segment joining $a(1)$ to the next marked point (or the previous marked point, respectively) along the boundary, see Figures \ref{fig:translationcompatible} and \ref{fig:translationnoncompatible}. We define $\Sigma_0^{-1}a$ and $\Sigma_1^{-1}a$ via the obvious inverse operations. 

		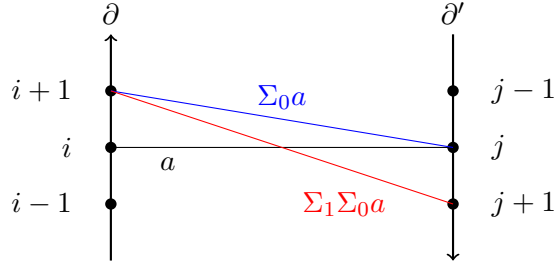
\begin{figure}
			\begin{center}
				\begin{tikzpicture}[scale = .75]
					\fill (-3,2) node [above] {$\d$};
					\fill (3,2) node [above] {$\d'$};
					\draw[thick,->] (-3,-2) -- (-3,2);
					\draw[thick,->] (3,2) -- (3,-2);

					\fill (-3,1) circle (.1);
					\fill (-3,0) circle (.1);
					\fill (-3,-1) circle (.1);

					\fill (-3.5,1) node [left] {$i+1$};
					\fill (-3.5,0) node [left] {$i$};
					\fill (-3.5,-1) node [left] {$i-1$};

					\fill (3,1) circle (.1);
					\fill (3,0) circle (.1);
					\fill (3,-1) circle (.1);

					\fill (3.5,1) node [right] {$j-1$};
					\fill (3.5,0) node [right] {$j$};
					\fill (3.5,-1) node [right] {$j+1$};

					\draw[] (-3,0) -- (3,0);
					\fill[] (-2,0) node [below] {$a$};

					\draw[blue] (-3,1) -- (3,0);
					\fill[blue] (0,.5) node [above] {$\Sigma_0 a$};

					\draw[red] (-3,1) -- (3,-1);
					\fill[red] (2,-1) node [left] {$\Sigma_1\Sigma_0 a$};
				\end{tikzpicture}
			\end{center}
			\caption{Actions of $\Sigma_0$ and $\Sigma_1$ for compatible orientations.}\label{fig:translationcompatible}
		\end{figure}

		\begin{figure}
			\begin{center}
				\begin{tikzpicture}[scale = .75]
					\fill (-3,2) node [above] {$\d$};
					\fill (3,2) node [above] {$\d$};
					\draw[thick,->] (-3,-2) -- (-3,2);
					\draw[thick,->] (3,2) -- (3,-2);

					\fill (-3,1) circle (.1);
					\fill (-3,0) circle (.1);
					\fill (-3,-1) circle (.1);

					\fill (-3.5,1) node [left] {$i+1$};
					\fill (-3.5,0) node [left] {$i$};
					\fill (-3.5,-1) node [left] {$i-1$};

					\fill (3,1) circle (.1);
					\fill (3,0) circle (.1);
					\fill (3,-1) circle (.1);

					\fill (3.5,1) node [right] {$j-1$};
					\fill (3.5,0) node [right] {$j$};
					\fill (3.5,-1) node [right] {$j+1$};

					\draw[black] (-3,0) -- (3,0);
					\fill[black] (-2,0) node [below] {$a$};

					\draw[blue] (-3,1) ..controls (-2.5,2) and (0,2) .. (0,0);
					\draw[blue] (3,0) ..controls (2.5,-2) and (0,-2) .. (0,0);
					\fill[blue] (1.75,-1.75) node [left] {$\Sigma_0 a$};

					\draw[red] (-3,1) ..controls (-2.5,.7) and (-.5,1) .. (-.1,0);
					\draw[red] (0,0) ..controls (1,-1) and (2,-1) .. (3,1);
					\fill[red] (2,-.85) node [left] {$\Sigma_1\Sigma_0 a$};

					\cc 0 0
				\end{tikzpicture}
			\end{center}
			\caption{Actions of $\Sigma_0$ and $\Sigma_1$ for non-compatible orientations.}\label{fig:translationnoncompatible}
		\end{figure}
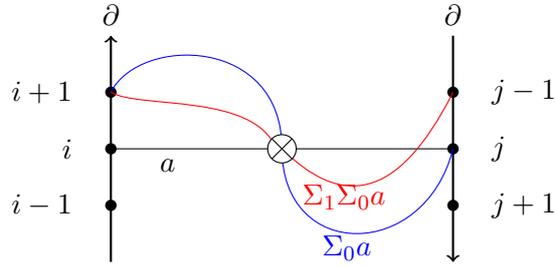

		\begin{defi}
			The \emph{AR-quiver} $\Gamma_{\mathcal H}$ is the oriented graph whose vertices are the elements of $\C_{\mathcal H}$ and whose arrows are given by
			$$\xymatrix@!R=3pt@!C=3pt{
				& \Sigma_0 a \ar[rd] \\
				a \ar[ru] \ar[rd] && \Sigma_1 \Sigma_0 a \\
				& \Sigma_1 a \ar[ru]
			}$$
		\end{defi}

		\begin{rmq}
			If $a:[0,1] \fl \S$ is an oriented curve, we denote by $a^{\op}$ the \emph{opposite curve} given by $a^{\op}(t) = a(1-t)$ for any $t \in [0,1]$. If $\mathcal H$ is a homotopy class of oriented curves in $\C\SM$, we denote by $\mathcal H^{\op}$ the homotopy class of the opposites of curves in $\mathcal H$. If $\d$ and $\d'$ are compatible with respect to $\mathcal H$, the AR-quiver $\Gamma_{\mathcal H}$ is independent on the choice of the orientation of curves in $\mathcal H$ so that the map sending $a$ to $a^{\op}$ yields an isomorphism $\Gamma_{\mathcal H^{\op}} \simeq \Gamma_{\mathcal H}$. But if $\d$ and $\d'$ are incompatible, then the map sending $a$ to $a^{\op}$  yields an isomorphism between $\Gamma_{\mathcal H^{\op}}$ and the opposite quiver $\Gamma_{\mathcal H}^{\op}$ of $\Gamma_{\mathcal H}$, that is the quiver with the same vertices but where every arrow is reversed.
		\end{rmq}

		\begin{defi}
			We set 
			$$\tau a = \Sigma_0^{-1}\Sigma_1^{-1}a = \Sigma_1^{-1}\Sigma_0^{-1}a \text{ and } \tau^{-1} a = \Sigma_0\Sigma_1 a = \Sigma_1\Sigma_0a$$
			and the map $\tau$ is called the \emph{AR-translation}.
		\end{defi}

		\begin{rmq}
			The terminology \emph{AR-quiver} stands for \emph{Auslander-Reiten quiver} since when $\SM$ is an orientable marked surface, the oriented graphs we just constructed describe connected components of the Auslander-Reiten quivers of the generalised cluster categories associated to the surface $\SM$, see \cite{CCS1,BZ:clustercatsurfaces}. In this case, the AR-translation defined above acts on $\Gamma_{\mathcal H}$ as the Auslander-Reiten translation functor on the corresponding connected component of the Auslander-Reiten quiver of the generalised cluster category.
		\end{rmq}

		We now prove that the AR-translation endows the AR-quiver of a homotopy class of curves with the structure of a stable translation quiver. We recall that a pair $(\Gamma, \tau)$ is called a \emph{stable translation quiver} if $\tau$ is a bijection from the set $\Gamma_0$ of vertices in $\Gamma$ to itself and if for any $a \in \Gamma_0$ it induces a bijection
		$$\ens{\alpha \in \Gamma_1 \ | \  \tau a \xrightarrow{\alpha} ?} \xrightarrow{\sim} \ens{\beta \in \Gamma_1 \ | \  ? \xrightarrow{\beta} a}$$
		where $\Gamma_1$ denotes the set of arrows in $\Gamma$. For generalities on translation quivers we refer the reader to \cite{Riedtmann:translation}.

		We denote by $\Z A_\infty^{\infty}$ the quiver whose vertices are labelled by $\Z \times \Z$ and with arrows $(i,j) \fl (i+1,j)$ and $(i,j) \fl (i,j+1)$ for any $i,j \in \Z$. It is a translation quiver for the translation given by $\tau (i,j) = (i-1,j-1)$, with $i,j \in \Z$.

		\begin{prop}\label{prop:quotientZA}
			Let $\SM$ be an unpunctured marked surface and $\mathcal H$ denote the homotopy class of curves in $\SM$ whose endpoints lie on boundary components of $\SM$. Then $\Gamma_{\mathcal H}$ is a stable translation quiver isomorphic to a quotient of $\Z A^{\infty}_{\infty}$ by a group of automorphisms.
		\end{prop}
		\begin{proof}
			Assume first that $\mathcal H$ does not contain any one-sided geodesic. Then, $\Gamma_{\mathcal H}$ is isomorphic as a translation quiver to a certain $\Gamma_{\mathcal H'}$ where $\mathcal H'$ is a homotopy class of curve in an orientable marked surface. Therefore, it follows from \cite{BZ:clustercatsurfaces} that $\Gamma_{\mathcal H}$ is isomorphic to a quotient of $\Z A_{\infty}^{\infty}$ by some automorphism group. . The only new case to treat is when $\mathcal H$ contains a one-sided geodesic. In this case, we simply observe that the natural action of the free group generated by $\Sigma_0$ and $\Sigma_1$ is free and transitive on $\C_{\mathcal H}$ so that $\Gamma_{\mathcal H} \simeq \Z A_{\infty}^{\infty}$.
		\end{proof}

	\subsection{A system of equations satisfied by $\lambda$-lengths}
		According to Proposition \ref{prop:quotientZA}, we can naturally label the vertices in $\Gamma_{\mathcal H}$ by couples $(i,j)$ with $i,j \in \Z$ with the convention the $g.(i,j)$ and $(i,j)$ label the same vertex for any $g$ in the automorphism group considered in Proposition \ref{prop:quotientZA}. Moreover, $\Sigma_0$ and $\Sigma_1$ act as
		$$\Sigma_0 (i,j) = (i+1,j) \text{ and }\Sigma_1 (i,j) = (i,j+\epsilon)$$
		for any $i,j \in \Z$ where $\epsilon = 1$ if $\omega$ and $\omega'$ are compatible with respect to $\mathcal H$ and $\epsilon = -1$ otherwise, see Figure \ref{fig:ARquiverH}.

		If $|\M \cap \d| = p$ and $|\M \cap \d'| = q$, we can label the marked points on $\d$ by $\Z/p\Z$ and the marked points on $\d'$ by $\Z/q\Z$ in such a way that the curve corresponding to the couple $(i,j)$ joins the marked point $i$ (modulo $p\Z$) in $\d$ to the marked point $j$ (modulo $q\Z$) in $\d'$.

		For any pair $(i,j) \in \Z \times \Z$, we denote by $\lambda^{\mathcal H}_{(i,j)}$ the $\lambda$-length of the curve in $\C_{\mathcal H}$ represented by the couple $(i,j)$. We also denote by $\lambda^{\d}_{\ens{i,i+1}}$ the $\lambda$-length of the boundary segment of $\d$ joining the marked point labelled by $i$ (modulo $p\Z$) to the marked point labelled by $i+1$ (modulo $p\Z$) and similarly for $\lambda^{\d'}_{\ens{j,j+1}}$. We adopt the convention that $\lambda^{\d}_{\ens{i,i}}=\lambda^{\d'}_{\ens{j,j}}=1$ for any $i,j \in \Z$.

		With these notations, it follows from the resolutions given in Theorem \ref{theorem:resolution} that for any $i,j \in \Z$, the $\lambda$-lengths of arcs in $\C_{\mathcal H}$ satisfy the following system of equations~:
		\begin{equation}\label{eq:system}
			\lambda^{\mathcal H}_{(i,j)}\lambda^{\mathcal H}_{(i+1,j+\epsilon)} = \lambda^{\mathcal H}_{(i+1,j)}\lambda^{\mathcal H}_{(i,j+\epsilon)}+\lambda^{\d}_{\ens{i,i+1}}\lambda^{\d'}_{\ens{j,j+\epsilon}}.
		\end{equation}
		or equivalently
		\begin{equation}\label{eq:systemtau}
			\lambda^{\mathcal H}_{(i,j)}\lambda^{\mathcal H}_{\tau^{-1} (i,j)} = \lambda^{\mathcal H}_{\Sigma_0(i,j)}\lambda^{\mathcal H}_{\Sigma_1(i,j)}+\lambda^{\d}_{\ens{i,i+1}}\lambda^{\d'}_{\ens{j,j+\epsilon}}.
		\end{equation}
		If we are in the case where $\d=\d'$ and curves in $\mathcal H$ are homotopic to the boundary $\d$, then these $\lambda$-lengths are moreover subject to the boundary conditions 
		$$\lambda^{\mathcal H}_{(i,i+1)}=\lambda^{\d}_{\ens{i,i+1}} \text{ and } \lambda^{\mathcal H}_{(i,i)}=1.$$

		\begin{rmq}
			In the ``coefficient-free'' settings, that is, when $\lambda$-lengths $\lambda^\d_{\ens{i,i+1}}$ and $\lambda^{\d'}_{\ens{j,j+1}}$ of boundary segments are specialised to 1, equation \eqref{eq:system} becomes 
			$$\lambda^{\mathcal H}_{(i,j)}\lambda^{\mathcal H}_{(i+1,j+1)} - \lambda^{\mathcal H}_{(i+1,j)}\lambda^{\mathcal H}_{(i,j+1)} =1$$
			so that each homotopy class of curves joining two boundary components in $\SM$ gives rise to a $SL_2$-tiling of the plane in the sense of \cite{ARS:frises}. 

			Equivalently, equation \eqref{eq:systemtau} becomes
			$$\lambda^{\mathcal H}_{a}\lambda^{\mathcal H}_{\tau^{-1} a} = \lambda^{\mathcal H}_{\Sigma_0 a}\lambda^{\mathcal H}_{\Sigma_1 a}+1$$
			for any $a \in \C_{\mathcal H}$. Thus, $\lambda$-length associated with curves in $\C_{\mathcal H}$ give rise to a frieze (in the sense of \cite{AD:algorithm}) on $\Z A_\infty^{\infty}$ with values in the ring of positive real-valued functions on the decorated Teichm\"uller space of $\SM$.
		\end{rmq}

	\subsection{Integration and partial triangulations}
		In this section we prove that the solutions of the systems \eqref{eq:system} are given by cluster variables in cluster algebras of type $A$ equipped with an alternating orientation. This allows to express the $\lambda$-lengths of curves in $\C_{\mathcal H}$ in terms of $\lambda$-lengths of a partial triangulation of $\SM$ consisting of arcs in $\C_{\mathcal H}$.

		Let $k \geq 1$ and $m \geq k-1$ be integers. We denote by $\Pi_m$ the disc with $m$ marked points on the boundary. Marked points are labelled cyclically by $\Z/m\Z$. Arcs in $\Pi_m$ are parametrised by pairs $\ens{i,j}$ with $i,j \in \Z/m\Z$ such that $i \neq j$ and $i \neq j \pm 1$. For such a pair $\ens{i,j}$, we denote by $x_{i,j}$ the corresponding cluster variables in $\mathcal A_{\Pi_m}$. For any $i \in \Z/m\Z$ we denote by $x_{i,i+1}$ the coefficient in $\mathcal A_{\Pi_m}$ corresponding to the boundary component $\ens{i,i+1}$. We consider the ``zig-zag'' triangulation of $\Pi_m$ given by arcs of the form $\ens{-i,i+2}$ and $\ens{-i,i+1}$ for $i \in \Z/m\Z$ (see Figure \ref{fig:zigzagPim} below). According to the Laurent phenomenon \cite{cluster1} and to the positivity conjecture for cluster algebras of type $A$ \cite{ST:unpunctured}, he variable $x_{0,k}$ can be written as a subtraction free Laurent polynomial in the coefficients and in the arcs of the zig-zag triangulation. More precisely, for any $k \geq 2$, there exists a unique 
		$$X_k \in \Z_{\geq 0}[x_{0,-1}, \ldots, x_{2-(k-1),2-k},x_{2,3}, \ldots, x_{k-1,k}][x_{0,2}^{\pm 1}, \ldots, x_{2-k,k}^{\pm 1},x_{-1,2}^{\pm 1}, \ldots, x_{2-(k-1),k}^{\pm 1}]$$
		such that $x_{0,k} = X_k$.

		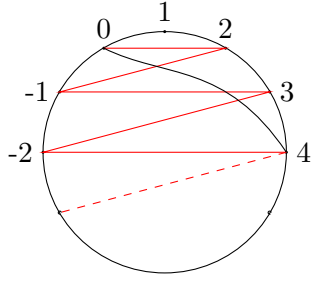
\begin{figure}
			\begin{center}
				\begin{tikzpicture}[scale=.4]

					\draw (0,0) circle (4);

					\fill (0,4) circle (.05);
					\fill (0,4) node [above] {1};

					\fill (-2,3.45) circle (.05);
					\fill (-2,3.45) node [above] {0};
					\fill (2,3.45) circle (.05);
					\fill (2,3.45) node [above] {2};

					\fill (-3.45,2) circle (.05);
					\fill (-3.45,2) node [left] {-1};
					\fill (3.45,2) circle (.05);	
					\fill (3.45,2) node [right] {3};
					
					\fill (-4,0) circle (.05);
					\fill (-4,0) node [left] {-2};
					\fill (4,0) circle (.05);
					\fill (4,0) node [right] {4};

					\draw[dashed] (-3.45,-2) circle (.05);
					\draw[dashed] (3.45,-2) circle (.05);

					\draw[red] (-2,3.45) -- (2,3.45) -- (-3.45,2) -- (3.45,2) -- (-4,0) -- (4,0);
					\draw[red,dashed] (4,0) -- (-3.45,-2);
			
					\draw (-2,3.45) .. controls (0,2.45) and (2,3) .. (4,0);
				\end{tikzpicture}
			\end{center}
			\caption{The zig-zag triangulation of $\Pi_m$.}\label{fig:zigzagPim}
		\end{figure}

		\begin{theorem}\label{theorem:formuleclose}
			Let $\SM$ be an unpunctured marked surface and let $\mathcal H$ be a homotopy class of curves joining the boundary components $\d$ and $\d'$. Then, for any $i,j \in \Z$ and any $k \geq 2$ we have~:
			\begin{align*}
			 \lambda^{\mathcal H}_{(i,j+k\epsilon)} 
				& = X_k\left(\lambda^{\d}_{\ens{i,i-1}}, \ldots, \lambda^{\d}_{\ens{i-(k-1),i-k}}, 
				\lambda^{\d'}_{\ens{j,j+\epsilon}}, \ldots, \lambda^{\d'}_{\ens{j+(k-1)\epsilon,j+k\epsilon}}, \right.\\
				& \left.\lambda^{\mathcal H}_{(i,j)}, \ldots, \lambda^{\mathcal H}_{(i-k,j+k\epsilon)}, \lambda^{\mathcal H}_{(i-1,j)}, \ldots, \lambda^{\mathcal H}_{(i-k,j+(k-1)\epsilon)}\right).
			\end{align*}
		\end{theorem}
		\begin{proof}
			Fix $k \geq 2$. With the previous notations, the AR-quiver $\Gamma_{\mathcal H}$ of $\mathcal H$ is of the form depicted in Figure \ref{fig:ARquiverH}.

			\begin{figure}
				\begin{center}
					\begin{tikzpicture}
						\tikzstyle{every node}=[font=\tiny]
						\foreach \x in {-2,0,2,4}
						{
							\foreach \y in {0,2,4,6}
							{
								\draw[->] (\x+.25,\y+.25) -- (\x+.75,\y+.75);
								\draw[->] (\x+.25,\y-.25) -- (\x+.75,\y-.75);
								\draw[->] (\x+1.25,\y+.25-1) -- (\x+1.75,\y+.75-1);
								\draw[->] (\x+.25+1,\y-.25+1) -- (\x+.75+1,\y-.75+1);
							}
						}
						
						\fill (-1,7) node {$\vdots$};
						\fill (1,7) node {$\vdots$};
						\fill (3,7) node {$\vdots$};
						\fill (5,7) node {$\vdots$};

						\fill (-1,-1) node {$\vdots$};
						\fill (1,-1) node {$\vdots$};
						\fill (3,-1) node {$\vdots$};
						\fill (5,-1) node {$\vdots$};

						\fill (-2,6) node {$\cdots$};
						\fill (-2,4) node {$\cdots$};
						\fill (-2,2) node {$\cdots$};
						\fill (-2,0) node {$\cdots$};

						\fill (-1,5) node {$\cdots$};
						\fill (-1,3) node {$\cdots$};
						\fill (-1,1) node {$\cdots$};

						\fill (6,6) node {$\cdots$};
						\fill (6,4) node {$\cdots$};
						\fill (6,2) node {$\cdots$};
						\fill (6,0) node {$\cdots$};

						\fill[red] (0,6) node {$(i+2,j-2\epsilon)$};
						\fill[red] (0,4) node {$(i+1,j-\epsilon)$};
						\fill[red] (0,2) node {$(i,j)$};
						\fill[red] (0,0) node {$(i-1,j+\epsilon)$};

						\fill (2,6) node {$(i+3,j-\epsilon)$};
						\fill (2,4) node {$(i+2,j)$};
						\fill (2,2) node {$(i+1,j+\epsilon)$};
						\fill (2,0) node {$(i,j+2\epsilon)$};

						\fill (4,6) node {$(i+4,j)$};
						\fill (4,4) node {$(i+3,j+\epsilon)$};
						\fill (4,2) node {$(i+2,j+2\epsilon)$};
						\fill (4,0) node {$(i+1,j+3\epsilon)$};

						\fill[red] (1,5) node {$(i+2,j-\epsilon)$};
						\fill[red] (1,3) node {$(i+1,j)$};
						\fill[red] (1,1) node {$(i,j+\epsilon)$};

						\fill (3,5) node {$(i+3,j)$};
						\fill (3,3) node {$(i+2,j+\epsilon)$};
						\fill (3,1) node {$(i+1,j+2\epsilon)$};

						\fill (5,5) node {$(i+5,j+\epsilon)$};
						\fill (5,3) node {$(i+3,j+2\epsilon)$};
						\fill (5,1) node {$(i+2,j+3\epsilon)$};

					\end{tikzpicture}
				\end{center}
				\caption{Local configuration in $\Gamma_{\mathcal H}$}\label{fig:ARquiverH}
			\end{figure}
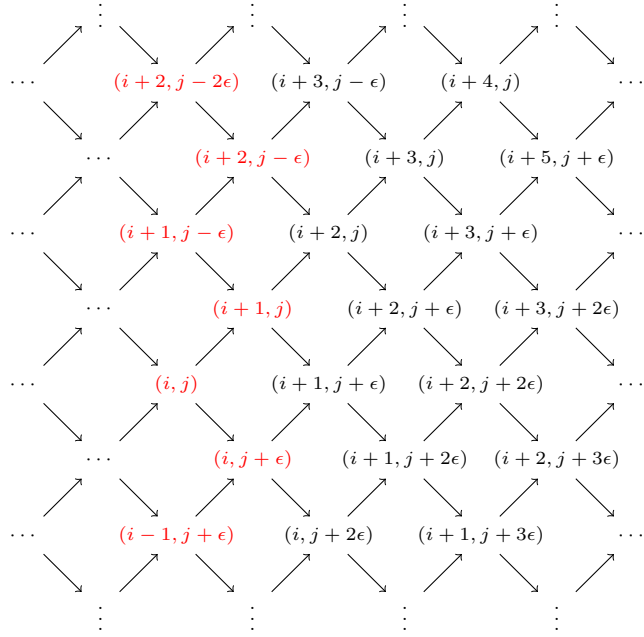

			Let $m \geq k-1$ and consider a morphism $\pi$ of $\Z$-algebras defined on $\mathcal A_{\Pi_m}$ and sending for any $0 \leq l \leq k-2$
			$$\begin{array}{rcl}
				x_{-l,-l-1}& \mapsto & \lambda^{\d}_{\ens{i-l,i-l-1}} \\
				x_{2+l,2+l+1}& \mapsto & \lambda^{\d'}_{\ens{j+l\epsilon,j+(l+1)\epsilon}} \\
				x_{-l,2+l} & \mapsto & \lambda^{\mathcal H}_{\ens{i-l,j+l\epsilon}} \\
				x_{-l-1,2+l} & \mapsto & \lambda^{\mathcal H}_{\ens{i-l-1,j+l\epsilon}}.
			\end{array}$$
			It is well-defined since variables corresponding to compatible arcs and boundary components are algebraically independent over $\Z$. Then, it follows directly from equations \eqref{eq:system} applied to both $\mathcal H$ and arcs in $\Pi_m$ that $\pi(x_{0,k}) = \lambda^{\mathcal H}_{(i,j+k\epsilon)}$, which proves the theorem.
		\end{proof}

		For any homotopy class $\mathcal H$ as above, we denote $\mathcal A_{\mathcal H}$ the subalgebra of $\mathcal A_{\SM}$ generated by the cluster variables $x_v$ where $v$ runs over the arcs in $\C_{\mathcal H}$. It follows from Theorem \ref{theorem:formuleclose} that each algebra $\mathcal A_{\mathcal H}$ is either a cluster algebra of type $A$ or is an infinite analogue of a cluster algebra of type $A$. Note that the algebra $\mathcal A_{\mathcal H}$ is independent on the choice of the orientations of the curves in $\mathcal H$.

		\begin{prop}\label{prop:presentation}
			Let $\SM$ be an unpunctured surface. Then the multiplication induces an epimorphism of $\Z$-algebras~:
			$$\Z\P \otimes \Z[x_{d} \ | \ d \in \qA\SM \setminus \A\SM] \otimes \left(\bigotimes_{\mathcal H} \mathcal A_{\mathcal H}\right) \fl \mathcal A_{\SM}$$
			where $\mathcal H$ runs over the possible homotopy classes of curves joining two marked points in $\SM$ and where tensor products are taken over the integers.
		\end{prop}
		\begin{proof}
			We first observe that the above mapping is a well-defined ring homomorphism since each term in the tensor product on the left-hand side is a sub-$\Z$-algebra of $\mathcal A_{\SM}$. Let $x \in \mathcal A_{\SM}$. By definition, we can write $x$ as a sum of terms of the form $b x_\delta x_\eta$ where $b \in \Z\P$, where $\delta$ is a multigeodesic consisting of one-sided simple closed curves and $\eta$ is a multigeodesic consisting of arcs in $\SM$. Let $\mathcal H_1, \ldots, \mathcal H_k$ be distinct homotopy classes of curves in $\SM$ such that $\eta = \eta_1 \sqcup \cdots \sqcup \eta_k$ where each geodesic in the multigeodesic $\eta_i$ belongs to $\mathcal H_i$. Therefore, $b x_\delta x_\eta$ is the image of $b \otimes x_\delta \otimes x_{\eta_1} \otimes \cdots \otimes x_{\eta_k}$ under the canonical mapping so that the mapping is surjective.
		\end{proof}

		\begin{rmq}
			Note that unless $\SM$ is a disc, the epimorphism given in Proposition \ref{prop:presentation} is not an isomorphism. Understanding the kernel of this map amounts to understanding the relations between $\lambda$-lengths of curves belonging to distinct homotopy classes. In the case of an annulus with all the boundaries specialised to 1, this situation was studied from a representation-theoretical point of view in \cite{AD:algorithm}. In this case, this amounts to compare the cluster characters associated to objects belonging to distinct connected components of the Auslander-Reiten quiver of a cluster category of type $\widetilde A$. 
		\end{rmq}

\providecommand{\bysame}{\leavevmode\hbox to3em{\hrulefill}\thinspace}
\providecommand{\MR}{\relax\ifhmode\unskip\space\fi MR }
\providecommand{\MRhref}[2]{%
  \href{http://www.ams.org/mathscinet-getitem?mr=#1}{#2}
}
\providecommand{\href}[2]{#2}

%

\end{document}